\documentclass[a4paper,12pt]{amsart} 

\usepackage[utf8]{inputenc}
\usepackage[T1]{fontenc}	
\usepackage[french,english]{babel}

\usepackage{lmodern}			
\usepackage{newtxtext}

\usepackage[top=3cm, bottom=3cm, left=3.6cm, right=3.6cm]{geometry}

\usepackage{graphicx}	

\usepackage{tikz}
\usetikzlibrary{patterns}

\usepackage[plainpages=false, colorlinks, linkcolor=bleuFonce, citecolor=rougeFonce, urlcolor=vertFonce, breaklinks]{hyperref}
\usepackage{orcidlink}

\usepackage{placeins}
\usepackage{float} 

\usepackage{color}
\definecolor{vertFonce}	{rgb}{0,0.5,0}
\definecolor{numLignes}	{rgb}{0.17,0.57,0.7}	
\definecolor{gris}		{rgb}{0.5,0.5,0.5}
\definecolor{grisFonce}	{rgb}{0.2,0.2,0.2}
\definecolor{orange}	{rgb}{1,0.65,0.31}		
\definecolor{orangeFonce}{rgb}{1,0.4,0}
\definecolor{bleuFonce}	{rgb}{0,0,0.4}
\definecolor{rougeFonce}{rgb}{0.3,0,0}
\definecolor{rougeWord}	{rgb}{0.5,0,0}
\definecolor{vertClair}	{rgb}{0.8,1,0.8}
\definecolor{rougeClair}{rgb}{1,0.5,0.5}
\definecolor{violet}	{rgb}{0.5,0,0.5}

\usepackage{pict2e}
\usepackage{multido}

\usepackage{amsfonts,amssymb,amsthm,amsmath} 
\usepackage[foot]{amsaddr}	
\usepackage{dsfont}				
\usepackage{mathrsfs}
\usepackage{yfonts}
\usepackage{cancel}
 
\theoremstyle{plain}
\newtheorem{thm}{Theorem}[section]
\newtheorem{lem}[thm]{Lemma}

\newtheorem{prop}[thm]{Proposition}

\theoremstyle{definition}

\newtheorem{remark}{Remark}[section]

\newenvironment{system*}{%
	\begin{equation*}\left\{\ \begin{aligned}
}{%
	\end{aligned} \right. \end{equation*}%
}

\newcommand		{\N}		{\mathbb N}			

\newcommand		{\RR}		{\mathbb R}			
\newcommand		{\R}		{\RR}
\newcommand		{\Rd}		{\R^3}
\newcommand		{\Rdd}		{\R^6}
\newcommand		{\hd}		{h^3}
\newcommand		{\CC}		{\mathbb C}			
\newcommand	{\cM}		{\mathcal M}		
\renewcommand	{\L}		{\mathcal L}		
\newcommand		{\cK}		{\mathcal K}		

\newcommand		{\lt}			{\left}				%
\newcommand		{\rt}			{\right}			%
\renewcommand	{\(}			{\lt(}
\renewcommand	{\)}			{\rt)}
\newcommand		{\set}[1]		{\lt\{#1\rt\}}
\newcommand		{\bangle}[1]	{\lt\langle #1\rt\rangle}
\newcommand		{\weight}[1]	{\bangle{#1}}	

\newcommand		{\com}[1]		{\lt[{#1}\rt]}		

\newcommand		{\n}[1]			{\lt\lvert{#1}\rt\rvert}	

\newcommand		{\sabs}[1]		{\lvert{#1}\rvert}

\newcommand		{\nrm}[1]		{\lt\lVert{#1}\rt\rVert}		
\newcommand		{\snrm}[1]		{\lVert #1\rVert}
\newcommand		{\bnrm}[1]		{\big\lVert #1\big\rVert}

\newcommand		{\Nrm}[2]		{\nrm{#1}_{#2}}
\newcommand		{\sNrm}[2]		{\snrm{#1}_{#2}}
\newcommand		{\bNrm}[2]		{\bnrm{#1}_{#2}}

\newcommand		{\indic}	{\mathds{1}}		

\renewcommand		{\d}		{\mathop{}\!\mathrm{d}}		
\newcommand			{\dpt}		{\partial_t}

\newcommand			{\Dx}		{\nabla_x}
\newcommand			{\Dv}		{\nabla_v}

\DeclareMathOperator{\pv}		{pv}
\DeclareMathOperator{\cF}		{\mathcal{F}}		
\DeclareMathOperator{\re}		{Re}				
\DeclareMathOperator{\im}		{Im}				
\DeclareMathOperator{\tr}		{Tr}				

\newcommand		{\F}[1]			{\cF\!\( #1 \)}		
\renewcommand	{\Re}[1]		{\re\!\( #1 \)}		
\renewcommand	{\Im}[1]		{\im\!\( #1 \)}		
\newcommand		{\Tr}[1]		{\tr\!\( #1 \)}		

\newcommand		{\intd}			{\int_{\Rd}}
\newcommand		{\intdd}		{\int_{\Rdd}}

\newcommand		{\init}			{\mathrm{init}}

\newcommand		{\eps}			{\varepsilon}

\newcommand	{\h}			{h}
\usepackage{braket}

\newcommand		{\opgam}		{{\boldsymbol{\gamma}}} 
\newcommand		{\op}			{\opgam}	

\newcommand		{\opmu}			{\boldsymbol{\mu}}	

\newcommand	 {\Wigner}[1]		{f^\textnormal{W}_{#1}}

\newcommand		{\opp}			{\boldsymbol{p}}

\newcommand		{\CK}			{\mathfrak{c}}	
\newcommand		{\Gr}			{G^{\textnormal{r}}}
\newcommand		{\GrE}			{G^{\textnormal{r},E}}
\newcommand		{\Gosc}			{G^{\textnormal{osc}}}
\newcommand		{\GoscE}		{G^{\textnormal{osc},E}}
\newcommand		{\hatG}			{\widehat{G}}
\newcommand		{\hatGr}		{\hatG^{\textnormal{r}}}
\newcommand		{\hatGosc}		{\hatG^{\textnormal{osc}}}
\newcommand		{\hatGrE}		{\hatG^{\textnormal{r},E}}
\newcommand		{\hatGoscE}		{\hatG^{\textnormal{osc},E}}

\title[\textsc{Quantum Landau Damping}]{\Large Quantum Landau Damping with Coulomb Repulsion in the Whole Space}

\author[\textsc{D. Bian}]{\vspace{-0pt}\large\textsc{Dongfen Bian}$^1$}
\address{\vspace{-0pt}$^1$School of Mathematics and Statistics \\ Beijing Institute of Technology, $100081$ Beijing, China}
\email[D.~Bian]{biandongfen@bit.edu.cn}

\author[\textsc{E. Grenier}]{\vspace{-0pt}\large\textsc{Emmanuel Grenier}$^2$}
\address{\vspace{-0pt}$^2$Academy of Mathematics and Systems Science, Chinese Academy of Sciences, Beijing 100190, China}
\email[E.~Grenier]{emmanuelgrenier@amss.ac.cn}

\author[\textsc{L. Lafleche}]{\vspace{-0pt}\large\textsc{Laurent Lafleche}$^3$}
\address{\vspace{-0pt}$^3$Unit\'e de Math\'ematiques Pures et Appliqu\'ees \\ CNRS and \'Ecole Normale Sup\'erieure de Lyon, UMR 5669 \\ 46, all\'ee d'Italie, 69364 Lyon Cedex 07, France}
\email[L.~Lafleche]{laurent.lafleche@ens-lyon.fr}

\author[\textsc{Q.H. Nguyen}]{\vspace{-0pt}\large\textsc{Quoc Hung Nguyen}$^2$}
\email[Q.-H.~Nguyen]{qhnguyen@amss.ac.cn}


\begin{document}

\begin{abstract}
	In this paper, we study the large time behavior of solutions to the linearized Hartree equation near a stable equilibrium which is homogeneous in space and investigate the quantum Landau damping.
	This allows us to get time decay of the electric field, uniformly in the Planck constant, as well as the uniform-in-time convergence of the quantum density towards the classical density as the Planck constant converges to $0$ .
\end{abstract}

\begingroup
\def\uppercasenonmath#1{} 
\let\MakeUppercase\relax 
\maketitle
\thispagestyle{empty} 
\endgroup

\bigskip





\section{Introduction}


	The aim of this article is to study the quantum analogue of the Landau damping. We recall that the linear Landau damping refers to the long-time behavior of solutions of the linearized Vlasov--Poisson equation
	\begin{equation}\label{eq:VP1}
		\dpt f + v \cdot \Dx f + E \cdot \Dv \mu = 0 \, ,
	\end{equation}
	with
	\begin{equation}\label{eq:VP2}
		E = - \nabla_x V\,, \qquad - \Delta_x V = \intd f \d v \, ,
	\end{equation}
	where $f(t,x,v)$ is the phase-space distribution of the electrons, $E(t,x)$ the electric field, $V(t,x)$ the electric potential, and $\mu(v)$ some homogeneous equilibrium. We assume that $\mu(v)$ is analytic, radial, and decaying in $\n{v}$. 

	In this case, it is well-known that the electric field $E(t,x)$ given by equations~\eqref{eq:VP1} and~\eqref{eq:VP2} converges to zero as $t \to + \infty$, a phenomenon known as ``Landau damping'', initially described by Lev Landau in his seminal paper~\cite{landau_vibrations_1946}. In particular, if $\widehat{E}(t,k)$ denotes the Fourier transform of the electric field with respect to the $x$ variable, then
	\begin{equation}\label{eq:damping}
		\sabs{\widehat{E}(t,k)} \lesssim e^{- \mathcal{D}(k) t}\, \sabs{\widehat{E}(0,k)}
	\end{equation}
	where $\mathcal{D}(k)$ is the Landau damping rate, which only depends on the equilibrium $\mu$. This damping rate vanishes for $k = 0$ and is of order $\n{k}$ for large $k$. In supremum norm, this gives a polynomial decay
	\begin{equation} \label{eq:damping2}
		\Nrm{E(t,\cdot)}{L^\infty} \lesssim \frac{1}{(1 + t)^{3/2}} \, .
	\end{equation}
	We refer to~\cite{landau_vibrations_1946} and~\cite{schekochihin_lectures_2023} for a detailed presentation from a physical point of view. On the mathematical side, the theory has first been focused to the linear case~\cite{maslov_linear_1986, degond_spectral_1986} and the existence of decaying solutions when the domain is a torus~\cite{caglioti_time_1998, hwang_existence_2009}. The nonlinear problem on the torus was later solved in~\cite{mouhot_landau_2011, bedrossian_landau_2016, grenier_landau_2021, ionescu_nonlinear_2024}. On the whole space, the mathematical analysis of the nonlinear problem is still incomplete. The asymptotic stability of homogeneous equilibria for screened interactions is studied in~\cite{bedrossian_landau_2018, han-kwan_asymptotic_2021}, and the linearized problem with Coulomb interaction is studied in~\cite{han-kwan_linearized_2021, bedrossian_linearized_2022, bian_linear_2024} where one can find a direct proof of inequalities~\eqref{eq:damping} and~\eqref{eq:damping2}.

	The aim of this article is to investigate the quantum counterpart of the Landau damping, namely to study the long-time behavior of the electric field corresponding to solutions of the linearized Hartree equation. The Hartree equation is the quantum analogue of the Vlasov equation, i.e. it is the mean-field equation describing the evolution of quantum particles. For a family of wave functions $\psi_j(t,x)$ such that $(\psi_j(t,\cdot))_{j\geq 0}$ form an orthonormal family in $L^2(\Rd)$, it reads
	\begin{equation}\label{eq:Hartree1}
		i\hbar\, \dpt \psi_j = -\,\frac{\hbar^2}{2}\, \Delta_x\psi_j + V_\h\,\psi_j
	\end{equation}
	where
	\begin{equation*}
		\hbar = \frac{h}{2\pi}
	\end{equation*}
	is the reduced Planck constant and the quantum electric potential $V_\h(t,x)$ is given by
	\begin{equation}\label{eq:Hartree2}
		-\Delta_x V_\h = \hd \sum_{j\geq 0} \lambda_j \n{\psi_j}^2
	\end{equation}
	for some summable family of positive numbers $(\lambda_j)_{j\in\N}$. We will in particular prove that inequalities~\eqref{eq:damping} and~\eqref{eq:damping2} remain true uniformly in $\h$ in the quantum case for 
	\begin{equation}\label{eq:Electric_field}
		E_\h := -\nabla_xV_\h
	\end{equation}
	for the linearized version of Equation~\eqref{eq:Hartree1}, and study how the damping rate depends on the Planck constant $h$. We will also prove that, in a sense given below, the solutions of the Hartree equation converge uniformly in time to the solutions of the Vlasov--Poisson equation in the semiclassical limit $h\to 0$. The question of the quantum analogue of the Landau damping was first studied in~\cite{klimontovich_spectra_1960, pines_approach_1962}, and we refer to~\cite{shukla_colloquium_2011} for a more comprehensive review of the physics literature on the subject. On the mathematical side, the uniform-in-$\h$ quantum Landau damping in the whole space was studied in the linear setting in dimension $1$ around a Lorentz equilibrium~\cite{gamba_note_2009} and was studied in the nonlinear case for screened interactions in~\cite{smith_phase_2024}. On the other hand, the limit from the Hartree to the Vlasov equation in the neighborhood of homogeneous equilibria was studied for smooth interactions in~\cite{lewin_hartree_2020}, and recently for screened interactions in~\cite{smith_semiclassical_2026}.

\subsection{The Hartree equation and the Wigner transform} 

	Before stating our results, we recall the standard quantum formalism, namely the operator form of the Hartree equation and its linearization, the Wigner transform and their links with the Vlasov equation.

\subsubsection{The Hartree equation}

	To make the parallel between the Hartree and the Vlasov equations clearer, it is useful to write the Hartree equation in its equivalent formulation in terms of operators. A density operator is a positive trace class operator $\op$ acting on $L^2(\Rd)$.
	Density operators can be written in the form of integral operators, that is for any $\varphi \in L^2(\Rd)$,
	\vspace{-1pt}\begin{equation*}
		\op\varphi(x) = \intd \op(x,y)\,\varphi(y)\d y
	\end{equation*}
	for some function $(x,y)\mapsto\op(x,y) \in L^2(\Rdd)$, which by abuse of notations is denoted by the same notation as the operator $\op$. By the spectral theorem, the specification of a density operator is equivalent to the specification of a family of orthonormal functions and a summable family of positive numbers, and therefore one can write the Hartree equation~\eqref{eq:Hartree1} as an equation for a time density dependent operator $\op(t)$ of the form
	\begin{equation}\label{eq:Hartree}
		i\hbar\,\dpt\op = \com{H_{\op}, \op}
	\end{equation}
	where $\com{A,B} := AB-BA$ denotes the commutator of two operators $A$ and $B$, and $H_{\op}$ is the Hartree Hamiltonian, given by 
	\vspace{-1pt}\begin{equation*}\vspace{-1pt}
		H_{\op} = \frac{\n{\opp}^2}{2} + V_\h
	\end{equation*} 
	where $\opp = -i\hbar\nabla$ and $\n{\opp}^2 = \opp^*\opp = -\hbar^2\Delta$, and where $V_\h$ is the operator of multiplication by the mean-field potential $V_\h(x)$. In terms of $\op$, $V_\h$ is defined by the formula
	\begin{equation}\label{eq:potential_quantum}
		V_\h(x) = K * \rho_\op(x) = \intd K(x-y)\, \rho_\op(y) \d y \, ,
	\end{equation}
	where $\rho_\op$ denotes the quantum position density, defined by
	\begin{equation*}
		\rho_\op(x) = \hd\, \op(x,x) \, ,
	\end{equation*}
	and $K$ is the Coulomb interaction kernel, defined by
	\begin{equation*}
		K(x) = \frac{\CK\,\pi}{\n{x}} \, ,
	\end{equation*}
	for some constant $\CK>0$, such as $\CK = (2\pi)^{-2}$ in the case of Equation~\eqref{eq:Hartree2}.

\subsubsection{The Wigner transform}

	To draw the link between quantum and classical theories, one can associate to an operator $\op$ a function of the phase space called its Wigner transform by the formula
	\begin{equation*} 
		\Wigner{\op}(x,v) = \intd e^{-i\,y\cdot v/\hbar} \,\op(x+\tfrac{y}{2},x-\tfrac{y}{2})\d y \, ,
	\end{equation*}
	which should be interpreted in general in the sense of distributions. It satisfies
	\begin{equation*}
		\Wigner{\op}(x,v) \d x\d v = \hd \Tr{\op}\,, \qquad \sNrm{\Wigner{\op}}{L^2(\Rdd)}^2 = \hd \Tr{\op^2}
	\end{equation*}
	and 
	\begin{equation*}
		\intd \Wigner{\op}(x,v) \d v = \rho_{\op}(x) \, .
	\end{equation*}
	The Wigner transform is a real-valued function if $\op$ is a self-adjoint operator, however it can have negative values even if $\op$ is a positive operator. If $\op$ satisfies the Hartree equation~\eqref{eq:Hartree}, then its Wigner transform satisfies (see e.g.~\cite{lions_sur_1993})
	\begin{equation*} 
		\dpt \Wigner{\op} + v\cdot\Dx \Wigner{\op} + \cK_\h \underset{v}{*} \Wigner{\op} = 0 \, ,
	\end{equation*}
	where
	\begin{equation}\label{eq:Wigner_Vlasov_Kh}
		\cK_\h(x,v) = 2i\pi \intd e^{-2i\pi y\cdot v}\, \frac{V_{\op}(x+hy/2)-V_{\op}(x-hy/2)}{h} \d y \, .
	\end{equation}
	from which one deduces at least formally that solutions of the Hartree equation converge in the semiclassical limit $\h\to 0$ to solutions of the Vlasov--Poisson equation. On the rigorous side, this is by now well-understood in the case of finite mass solutions, see~\cite{lions_sur_1993, gerard_homogenization_1997, ambrosio_semiclassical_2011} for the case of weak convergence of rather general data, and \cite{lafleche_propagation_2019, lafleche_strong_2023, chong_l2_2023, iacobelli_enhanced_2024} for quantitative results in stronger norms for more regular data. In the case of solutions with infinite mass, the derivation of the Vlasov--Poisson equation from the Hartree equation~\eqref{eq:Hartree} is however not known, although it is known for smooth potentials~\cite{lewin_hartree_2020}.

\subsubsection{Perturbation of an homogeneous equilibrium}

	In this article, we consider the case where $\op$ is a small perturbation of an homogeneous equilibrium
	\begin{equation}\label{eq:quantum_equilibrium}
		\opmu = \mu(\opp) = \mu(-i\hbar\nabla)
	\end{equation}
	for some probability distribution $\mu$ which is radial and decreasing. The operator $\opmu$ is a positive Fourier multiplier with integral kernel 
	\begin{equation*}
		\opmu(x,y) = \frac{1}{\hd} \,\widehat{\mu}\!\lt(\frac{y-x}{h}\rt),
	\end{equation*}
	where we take the following convention for the Fourier transform
	\begin{equation*}
		\widehat{u}(k) = \F{u}\!(k) = \intd e^{-2i\pi x\cdot k} \,u(x) \d x \, .
	\end{equation*}
	In particular, $\opmu$ is not a trace class operator and its Wigner transform is just
	\begin{equation*}
		\Wigner{\opmu}(x,v) = \frac{1}{\hd} \intd e^{-i\,y\cdot v/\hbar} \,\widehat{\mu}(-y/h)\d y = \intd e^{2i\pi\,y\cdot v} \,\widehat{\mu}(y)\d y = \mu(v)
	\end{equation*}
	by the Fourier inversion theorem. As a consequence, $\rho_{\opmu} = \intd \mu \d v = 1$ and $V_{\opmu} = \intd K$ is constant, therefore, if $\op_\eps = \opmu + \eps\,\op$ is a solution of the Hartree equation, then
	\begin{equation*} 
		i\hbar\, \dpt \op = \frac{1}{2}\com{\n{\opp}^2,\op} + \com{V_{\op},\opmu} + \eps \com{V_{\op},\op} .
	\end{equation*}
	Neglecting the term $\eps \com{V_{\op},\op}$, we obtain the linearized Hartree equation
	\begin{equation}\label{eq:Hartree_linear}
		i\hbar\, \dpt \op = \frac{1}{2}\com{\n{\opp}^2,\op} + \com{V_{\op},\opmu} ,
	\end{equation}
	which is the quantum analogue of the linearized Vlasov--Poisson equation~\eqref{eq:VP1}. The associated linearized Wigner equation then reads
	\begin{equation}\label{eq:linear_Wigner}
		\dpt \Wigner{\op} + v\cdot\Dx \Wigner{\op} + \cK_\h \underset{v}{*} \mu = 0 \, ,
	\end{equation}
	where $\cK_\h$ was defined in Equation~\eqref{eq:Wigner_Vlasov_Kh}.

\section{Statement of the main results}
	In this paper, we study the semiclassical limit and the long time behavior of solutions to the linearized Wigner equation~\eqref{eq:linear_Wigner} and in particular the so-called Landau damping. We assume that $\mu$ is analytic, radial, strictly decreasing as a function of $\n{v}$, and that, for every multi-index $\beta$ such that $\n{\beta}\leq 8$,
	\begin{equation*}
		\sabs{\partial_v^\beta\mu(v)} \lesssim \weight{v}^{-8-\n{\beta}} .
	\end{equation*}
	Here $\weight{x} = \sqrt{1+\n{x}^2}$, and $a\lesssim b$ means $a\leq C\,b$ with $C$ independent of $t$ and $h$.\\
	
	Let $F(s)$ be defined for any $s\in\R$ by
	\begin{equation}\label{eq:marginal}
		F(s) := \int_{\R^2}\mu(s,v_2,v_3)\d v_2\d v_3 \, .
	\end{equation}
	We assume that there exists $\eta\in(0,1/2]$ such that $F$ extends
	holomorphically to the conic neighborhood
	\begin{equation}\label{eq:def_conic}
		\mathcal D_\eta
		:=\Set{z\in\CC:\n{\Im z}<\eta\lt(1+\n{\Re z}\rt)},
	\end{equation}
	and satisfies, for every $z\in\mathcal D_\eta$,
	\begin{equation}\label{eq:main-assumption-marginal}
		F(-z)=F(z),
		\qquad F(\overline z)=\overline{F(z)},
		\qquad \n{F(z)}\lesssim\weight{z}^{-6}.
	\end{equation}
	In addition, $F(s)>0$ and $F'(s)<0$ for $s>0$. Thus $F$ is a
	thin-tail acceptable equilibrium, with decay exponent $d=6$, in the
	sense of~\cite[Definition~1.1]{ionescu_stability_2023}.
	Typical examples satisfying these assumptions are
	\begin{equation}\label{eq:examples-admissible-equilibria}
		\mu_{\rm G}(v) = c_{\rm G}\, e^{-\n{v}^2} ,
		\quad
		\mu_{\rm E}(v) = c_{\rm E} \, e^{-\weight{v}} ,
		\quad
		\mu_\beta(v) = c_\beta\weight{v}^{-3-\beta} \text{ with }\beta\geq 5
	\end{equation}
	where the positive constants may be chosen to normalize the mass. Their
	one-dimensional marginals are, up to positive multiplicative constants,
	\begin{equation*}
		F_{\rm G}(s) = e^{-s^2},
		\qquad
		F_{\rm E}(s) = \lt(1+\weight{s}\rt)e^{-\weight{s}} ,
		\quad
		F_\beta(s) = \weight{s}^{-1-\beta},
	\end{equation*}
	For $\beta\geq5$, one has
	$F_\beta(s)\lesssim\weight{s}^{-6}$, while the other two marginals
	have exponential decay. After choosing $\eta>0$ sufficiently small,
	each marginal extends holomorphically to $\mathcal D_\eta$, obeys
	\eqref{eq:main-assumption-marginal}, and is strictly decreasing on
	$(0,\infty)$. 

	To simplify the statements, we assume that $\CK = 1$ and
	\begin{equation*}
		M_0 := \intd \mu(v) \d v = 1\,,
		\qquad
		M_2 := \intd \frac{\n{v}^2}{2} \,\mu(v) \d v = 1 \, 
	\end{equation*}
	which is always possible, up to a rescaling in velocity.
	
	By convention, $h = 0$ refers to the classical case, namely to solutions of the Vlasov--Poisson equation or to Green functions related to it.

	\begin{thm}\label{thm:simplified}
		Let $f_0$ be a solution of Equation~\eqref{eq:VP1} with initial data $f_0^\init \in L^1$ and $\op$ be a solution of Equation~\eqref{eq:Hartree_linear} with initial data $\op^\init$ a density operator with the same initial density $\rho_\h^\init := \rho_{\op^\init} = \rho_0^\init \in L^1$. Then, uniformly in $0 \le h \le 1$, the associated electric field defined in Equation~\eqref{eq:Electric_field} or~\eqref{eq:VP2} satisfies
		\begin{equation}\label{eq:decay_E}
			\Nrm{E_\h(t,\cdot)}{L^\infty(\Rd)} \lesssim \weight{t}^{-3/2}
		\end{equation}
		and uniformly in $0 \le h \le 1$ and $t\in\R_+$, the position density satisfies
		\begin{equation}\label{eq:limit_rho}
			\Nrm{\rho_\h(t,\cdot) - \rho_0(t,\cdot)}{L^\infty} \lesssim h^2.
		\end{equation}
		If $f_0^\init \in L^1_xL^{6/5}_v$ and $\hd \tr\bigl(\sabs{\weight{x}^{3/2}\op^\init}^2+\sabs{\weight{x}^{3/2}\op^\init}^2\bigr)$ is bounded uniformly in $\hbar$, then the Wigner transform $f_\h$ of $\op$ satisfies
		\begin{equation}\label{eq:limit_f}
			\Nrm{f_\h(t,\cdot) - f_0(t,\cdot)}{L^2} \leq 
			\sNrm{f_\h^\init - f_0^\init}{L^2} + C\, h^2 \weight{t} .
		\end{equation}
	\end{thm}

	In particular, this shows that, in this case. the semiclassical approximation is valid for times of order $T = o(\h^{-2})$, much longer than the usual Ehrenfest time $T_\textnormal{E} \sim \ln(1/\h)$. 
	
	The proof of this theorem relies on a detailed study of the Green functions of the linearized Wigner equations and of the linearized Vlasov--Poisson equation. The theorem is consequence of the Lemmas~\ref{lem:decay_Gosc} and~\ref{lem:decay_Gr} detailed below. We first describe the Green function of the linearized Wigner equation.

	\begin{thm}\label{thm:Green_functions}
		Let $h\in[0,1]$, $\op$ be a solution of Equation~\eqref{eq:Hartree_linear} with initial data $\op^{\init}$ a density operator. Let $f_h^\circ(t,x,v)$ be the solution of the free transport equation with initial data $f_h^\init := \Wigner{\op^\init}$ the Wigner transform of $\op^\init$, namely
		\begin{equation*}
			f_h^\circ(t,x,v) = f_h^\init(x - tv,v) \, .
		\end{equation*}
		Let $\rho_\h := \rho_\op$ and $\rho_\h^\circ$ be the corresponding densities
		\begin{equation*}
			\rho_\h(t,x) = \intd \Wigner{\op}(t,x,v) \d v\, , \qquad \rho_\h^\circ(t,x) = \intd f_h^\circ(t,x,v) \d v \, .
		\end{equation*}
		Then, one can decompose $\rho_\h$ in the following way
		\begin{equation*} 
			\rho_\h = \Gosc_{\h,+} \star_{t,x} \rho_\h^\circ + \Gosc_{\h,-} \star_{t,x} \rho_\h^\circ + \Gr_\h \star_{t,x} \rho_\h^\circ ,
		\end{equation*}
		where $\Gosc_{\h,\pm}$, the oscillatory Green functions, are defined by their Fourier transforms
		\begin{equation*}
			\hatGosc_{\h,\pm}(t,k) = a_{\h,\pm}(k)\, e^{\lambda_{\h,\pm}(k) \,t}
		\end{equation*}
		with, $a_{\h,\pm}(k) = 0$ if $\n{k} \ge 1$, $\lambda_{\h,-}(k) = \bar \lambda_{\h,+}(k)$, $a_{\h,-}(k) = \bar a_{\h,+}(k)$, and uniformly in $\h$, as $k \to 0$,
		\begin{align*}
			a_{\h,\pm}(k) &= \pm \frac{i}{2} + O(\n{k}^2) \, ,
			\\
			\im \lambda_{\h,\pm}(k) &= \pm 1 \mp \n{k}^2 + O(\n{k}^4) \, ,
			\\
			\re \lambda_{\h,\pm}(k) &\sim - \frac{2 \pi}{\n{k}^2} \int_{\mathbb{R}^2} \partial_{v_1} \mu\!\lt( \n{k}^{-1},v_2,v_3 \rt) \d v_2 \d v_3 \, ,
		\end{align*}
		and where, uniformly in $\h$,
		\begin{equation*} 
			\n{\hatGr_\h(t,k)} \lesssim {|k|^3 \over \langle k \rangle^4} e^{- \sigma \n{k} t} 
		\end{equation*}
		for some $\sigma > 0$.
	\end{thm}

	\begin{remark}
		The free transport density is the density associated to the free Schrödinger equation, that is
		\begin{equation*}
			\rho_\h^\circ(t,x) = \hd \bigl(e^{\frac{it \hbar}{2}\Delta}\, \op^\init e^{\frac{-it \hbar}{2}\Delta}\bigr)(x,x) \, .
		\end{equation*}
		In particular, if $\hd \, \op^\init$ is the $L^2$ projection on a normalized wave function $\psi \in L^2(\Rd)$, then $\rho_\h^\circ(t,x) = \sabs{e^{\frac{it \hbar}{2}\Delta}\psi(x)}^2$.
	\end{remark}

	We now compare the Green function of the linearized Wigner equation with the Green function of the linearized Vlasov--Poisson equation.

	\begin{thm}\label{thm:Green_functions_limit}
		Let $f_0(t,x,v)$ be the solution of the linearized Vlasov equation with initial data $f_0^{\init}$. Let $\rho_0(t,x)$ be the corresponding density. Let $f_0^\circ$ and $\rho_0^\circ$ be defined as in Theorem~\ref{thm:Green_functions}. Then 
		\begin{equation*} 
			\rho_0 = \Gosc_{0,+} \star_{t,x} \rho_0^\circ + \Gosc_{0,-} \star_{t,x} \rho_0^\circ + \Gr_0 \star_{t,x} \rho_0^\circ\, ,
		\end{equation*}
		where $\Gosc_{0,\pm}$ and $\Gr_0$ have the same structure as $\Gosc_{\h,\pm}$ and $\Gr_\h$ and satisfy the same bounds.
		Moreover,
		\begin{equation*}
			\n{\lambda_{\h,\pm}(k) - \lambda_{0,\pm}(k)} \lesssim \n{k}^4 h^2,
		\end{equation*}
		\begin{equation*}
			\n{a_{\h,\pm} - a_{0,\pm}} \lesssim \n{k}^4 h^2,
		\end{equation*}
		and 
		\begin{equation*}
			\n{\hatGr_\h(t,k) - \hatGr_0(t,k)} \lesssim \lt( {|k|^3 \over \langle k \rangle^4} \rt) e^{- \sigma \n{k} t} \min(1,h|k|)^2 .
		\end{equation*}
	\end{thm}

	The Green functions $\Gosc_{h,\pm}$ and $\Gr_h$ are studied in detail in section \ref{physicalspace}.

\medskip


\section{Fourier transform of the Wigner equation}


	Our strategy will be to work mainly on the Wigner equation~\eqref{eq:linear_Wigner} and notice that it is very similar to the Vlasov equation~\eqref{eq:VP1} when taking the Fourier transform in the $x$ variable of both equations.

	If $k \in \Rd$, we set
	\begin{equation*}
		\frac{1}{k} := \frac{k}{\n{k}^2}\, .
	\end{equation*}
	The function $\cK_\h$ defined in Equation~\eqref{eq:Wigner_Vlasov_Kh} then satisfies
	\begin{equation*}
		\cF_x\lt(\cK_\h\underset{v}{*} \mu\rt)(k,v) = \widehat{E_\h}(k)\cdot \frac{\mu(v+\tfrac{h k}{2}) - \mu(v-\tfrac{h k}{2})}{hk} = \widehat{E_\h}(k)\cdot D_{h k}\mu(v) \, ,
	\end{equation*}
	where $E_\h = - \nabla V_\h$ with $V_\h$ defined in Equation~\eqref{eq:potential_quantum}, and where we defined the discrete gradient $D_q \mu$ of $\mu$ by
	\begin{equation*}
	D_q\mu(v) = \frac{\mu(v+\tfrac{q}{2}) - \mu(v-\tfrac{q}{2})}{q} \, .
	\end{equation*}
	We observe that $D_q \mu(v)$ is uniformly bounded in $q > 0$ and converges to $\frac{q}{\n{q}}\cdot D_0 \mu$ as $q$ goes to $0$, where $D_0\mu := \nabla_v \mu$ is the usual gradient in the velocity variable. We note that
	\begin{equation*}
		D_q \mu(v) = D_0 \mu(v) + O(q^2) \, .
	\end{equation*}
	From the above computation, taking the Fourier transform of the linear Wigner equation~\eqref{eq:linear_Wigner} in the $x$ variable, with dual Fourier variable $k$, gives
	\begin{equation}\label{eq:transfo}
		\dpt \cF_x(\Wigner{\op}) + 2i\pi k\cdot v \,\cF_x(\Wigner{\op}) + \widehat{E_\h}(k)\cdot D_{h k}\mu(v) = 0,
	\end{equation}
	while the same Fourier transform of the Vlasov equation leads to
	\begin{equation}\label{eq:Vlasov_Fourier}
		\dpt \cF_x(f) + 2i\pi k\cdot v \,\cF_x(f) + \widehat{E}(k)\cdot \Dv\mu(v) = 0 \, .
	\end{equation}
	In particular, the classical case~\eqref{eq:Vlasov_Fourier} appears to be the limit of the quantum case~\eqref{eq:transfo} as $\h$ goes to $0$, the main difference between~\eqref{eq:transfo} and~\eqref{eq:Vlasov_Fourier} being that the usual gradient is replaced by the discrete one in the quantum equation.


\section{Study of the dielectric function}


	In this section, we study the dielectric function to help us take into account the effect of the interactions on the position density. The proof follows the strategy of~\cite{bedrossian_linearized_2022, bian_linear_2024, degond_spectral_1986}.


\subsection{Construction of the dielectric function}


	We look at the equation
	\begin{equation*}
		\dpt \cF_x(f) + 2i\pi k\cdot v \,\cF_x(f) + \widehat{E_\h}(k)\cdot D_{h k}\mu(v) = 0 \, ,
	\end{equation*}
	with $f = \Wigner{\op}$ if $h>0$ and $f$ a solution of the Vlasov equation and $D_0 = \Dv$ if $h=0$. We write $E$ the force field corresponding to $f$. We know take the Laplace transform in time with dual variable $\lambda$, i.e. we define
	\begin{equation*}
		\widetilde{f}(\lambda,k,v) = \int_0^\infty\!\!\!\intd e^{-2\pi\(\lambda t + i x\cdot k\)} f(t,x,v)\d x \d t
	\end{equation*}
	which yields
	\begin{equation*}
		2\pi\lt(\lambda + i k\cdot v\rt)\widetilde{f} + \widetilde{E}(k)\cdot D_{h k}\mu(v) = \widehat{f}^\init,
	\end{equation*}
	where $f^\init(x,v) = f(0,x,v)$. Assume that $\Re{\lambda}$ is positive and large enough, so that it is on the right of the spectrum of the linearized Vlasov--Poisson equation. Since 
	\begin{equation*}
	\widetilde{E}(k) = -2i\pi\, k\,\widehat{K}(k)\, \widetilde{\rho}(\lambda,k) \, ,
	\end{equation*}
	dividing by $ 2\pi \lt(\lambda + i k\cdot v \rt)$ and integrating over $v$ gives
	\begin{equation} \label{division}
		\epsilon_{\mu,\h}(\lambda,k)\, \widetilde{\rho}(\lambda,k) = \widetilde{\rho^\circ}(\lambda,k)
	\end{equation}
	where
	\begin{equation*}
		\widetilde{\rho^\circ}(\lambda,k) = \frac{1}{2\pi}\intd \frac{\widehat{f}^\init(k,v)}{\lambda + i k\cdot v} \d v
	\end{equation*}
	and where the dielectric function $\epsilon_{\mu,\h}(\lambda,k)$ is defined by 
	$\epsilon_{\mu,\h}(\lambda,k) = \epsilon_{\mu,\h}^+(\lambda,k)$ with
	\begin{equation*}
		\epsilon_{\mu,\h}^+(\lambda,k) = 1 - \frac{1}{2\pi}\intd \frac{\widehat{\nabla K}(k)\cdot D_{h k}\mu(v)}{\lambda + i k\cdot v} \d v \, .
	\end{equation*}
	The inverse Laplace formula gives that for $\Lambda\in\R$ large enough
	\begin{equation*}
		\widehat{\rho}(t,k) = \frac{1}{i} \int_\Gamma \frac{\widetilde{\rho}(\lambda,k)}{\epsilon_{\mu,\h}(\lambda,k)} \,e^{2\pi\lambda t}\d\lambda
	\end{equation*}
	where $\Gamma = \Lambda + i\R$. 

	Let us first comment the term $\rho^\circ$. Let $f^\circ$ be the solution of the free transport 
	\begin{equation*}
		\partial_t f^\circ + v \cdot \nabla_x f^\circ = 0
	\end{equation*}
	with initial data $f^\init(x,v)$. Then its Fourier--Laplace transform $\tilde f^\circ$ equals $(2\pi)^{-1}(\lambda + i\, k \cdot v)^{-1} \hat f^\circ$. Thus $\widetilde{\rho^\circ}$ is exactly the Fourier--Laplace transform of the density $\rho^\circ$ created by $f^\circ$, that is
	\begin{equation*}
		\rho^\circ(t,x) = \intd f^\circ(t,x,v) \d v = \intd f^\init(x-t\,v,v)\d v\, .
	\end{equation*}
	
	Let us now study the dielectric function. We observe that $\epsilon_{\mu,\h}(\lambda,k)$ defines a holomorphic function on the half-plane $\Re{\lambda}>0$, since then, the denominator never vanishes. 
	For simplicity, let us assume that $k = (k_1,0,0)$ with $k_1\geq 0$. Then for $\Re{\lambda}>0$ and with $v=(v_1,v_2,v_3)$, we get
	\begin{equation*}
		\epsilon_{\mu,\h}(\lambda,k) = 1 - \frac{1}{2\pi}\intd \frac{\widehat{\nabla K}(k)\, D_{h k} \mu(v)}{\lambda + i\,k_1\,v_1} \d v\, .
	\end{equation*}
	We note that
	\begin{equation*}
		\epsilon_{\mu,\h}(\lambda,k) = \epsilon_{\mu,0}(\lambda,k) + O(\n{k}^2h^2) \, .
	\end{equation*}
	We now extend the dielectric function to $\Re \lambda \le 0$. The distribution $(\lambda + i\, x)^{-1}$ of the variable $x$ for $\Re{\lambda} > 0$ can be analytically extended to $(\lambda + i\, x)^{-1} + 2 i \pi\, \delta_{i\lambda}$ for negative values of $\Re{\lambda}$ and to $\pv(1/ix) + i \pi \,\delta_{i\lambda}$ for $\eps = 0$, where $\pv$ denotes the principal value. Thus, we can extend the dielectric function as follows
	\begin{align}
		\epsilon_{\mu,\h}(\lambda,k) &= \epsilon_{\mu,\h}^+(\lambda,k) + i\,\widehat{K}(k) \,\epsilon_{\mu,\h}^0(\lambda,k)
		& \text{ if } \Re{\lambda} < 0,
		\\\label{eq:eps_zero}
		\epsilon_{\mu,\h}(\lambda,k) &= \epsilon_{\mu,\h}^\text{pv}(\lambda,k) + \frac{i}{2}\,\widehat{K}(k) \,\epsilon_{\mu,\h}^0(\lambda,k)
		& \text{ if } \Re{\lambda} = 0,
	\end{align}
	where
	\begin{align*}
		\epsilon_{\mu,\h}^\text{pv}(\lambda,k) 
		&= 1 - \lim_{\eps\to 0}\frac{1}{2\pi}\intd \frac{\widehat{\nabla K}(k)\cdot D_{h k}\mu(v)}{\lambda + i k\cdot v}\, \indic_{\n{\lambda + ik\cdot v}\geq \eps}\d v,
		\\
		\epsilon_{\mu,\h}^0(\lambda,k) &= 2\pi\int_{\R^2} D_{h k} \mu(\vartheta,v_2,v_3) \d v_2 \d v_3 
	\end{align*}
	with $\vartheta = i \lambda/|k|$ (we refer the reader to \cite[p.~26]{schekochihin_lectures_2023}, for a detailed discussion). We note that
	\begin{equation*}
		\epsilon_{\mu,\h}^\text{pv}(\lambda,k) = \epsilon_{\mu,0}^\text{pv}(\lambda,k) + O(\n{k}^{2\theta} h^{2\theta})
	\end{equation*}
	and
	\begin{equation*}
		\epsilon_{\mu,\h}^0(\lambda,k) = \epsilon_{\mu,0}^0(\lambda,k) + O(\n{k}^{2\theta} h^{2\theta}) \, ,
	\end{equation*}
	for any $\theta\in[0,1]$.


\subsection{Zeros of the dielectric function when \texorpdfstring{$\Re\lambda \ge 0$}{lambda has a nonnegative real part}}


	This section is devoted to the study of the so-called dispersion relation $\epsilon_{\mu,\h}(\lambda,k) = 0$. The following computations are well-known in physics in the case of the Vlasov--Poisson equation, and can be found, for instance, in~\cite[p.~43]{schekochihin_lectures_2023}. In particular, it is known that if $\mu$ is radially decreasing, then there is no unstable mode. More precisely, the following lemma holds.
		
	\begin{lem}
		If $\mu$ is radial and (strictly) decreasing , then $\epsilon_{\mu,\h}(\lambda,k)$ does not vanish when $\Re \lambda \geq 0$.
	\end{lem}
	
	\begin{proof}
		Let us start with the case $\Re \lambda >0$. The proof uses Nyquist's method. Let $\Gamma$ be a contour in $\Re \lambda > 0$. As $\epsilon_{\mu,\h}(\lambda,k)$ has no pole in this half place,
		\begin{equation*}
			\frac{1}{2 i \pi} \int_\Gamma \partial_\lambda \ln \epsilon_{\mu,\h}(\lambda,k) \d\lambda
		\end{equation*}
		is the number of zeros of $\epsilon_{\mu,\h}$, counted with their multiplicities, which are included in $\Gamma$. We choose for $\Gamma$ the contour $i \R + \eps$ with $\eps > 0$, closed by a circle "at infinity". As $\epsilon_{\mu,\h} \to 1$, when $\n{\lambda} \to + \infty$, we may let $\eps \to 0^+$. The number $N$ of unstable eigenvalues in the half plane $\Re \lambda > 0$ is thus given by
		\begin{equation*}
			2 i \pi N = \int_{+ i \infty + 0}^{- i \infty + 0} \partial_\lambda \ln \epsilon_{\mu,\h}(\lambda,k) \d\lambda \, .
		\end{equation*}
		We thus have to count the number of times $\epsilon_{\mu,\h}(\lambda,k)$ circles around the origin.
		But when $\Re{\lambda} = 0$, $\partial_\lambda \epsilon_{\mu,\h}$ is given by \eqref{eq:eps_zero} and since $\widehat{K}$ is real-valued, it is real only when $\epsilon_{\mu,\h}^0(\lambda,k)$ vanishes, namely only at $\lambda = 0$. But
		\begin{equation*}
			\epsilon_{\mu,\h}^\text{pv}(0,k) = 1 + \widehat{K}(k) \lim_{\eps\to 0}\intd \frac{ \mu(v-\frac{hk}{2}) - \mu(v+\frac{h k}{2})}{h\,k_1\,v_1}\,\indic_{\n{v_1}\geq \eps}\d v
		\end{equation*}
		where $k = (k_1,0,0)$ with $k_1>0$. As $\mu$ is radial and strictly decreasing,
		\begin{equation*}
			\frac{\mu(v_1-\tfrac{hk_1}{2},v_2,v_3) - \mu(v_1+\tfrac{h_1k}{2},v_2,v_3)}{v_1} > 0 \, ,
		\end{equation*}
		both for positive and negative $v_1$,
		More precisely, for $v_1>0$ fixed, the function $k_1\mapsto \nu(v_1-\tfrac{hk_1}{2}) - \nu(v_1+\tfrac{h_1k}{2})$ is strictly positive on $\R$ and continuous, hence we can define for any $r>0$,
		\begin{equation}\label{eq:bound_convexity_mu}
			\nu_r(v) := \inf_{\n{q}\leq r} (- D_q\mu(v)) > 0 \, .
		\end{equation}
		The same holds if $v_1<0$, hence $\epsilon_{\mu,\h}^\text{pv}(0,k) > 0$. In particular $\epsilon_{\mu,\h}(\lambda,k)$ does not circle the origin when $\lambda$ describes $i \R$. Thus $N = 0$, which ends the proof in the case $\Re \lambda > 0$.
		
		In the case $\Re{\lambda} = 0$, then we use again the fact that $\epsilon_{\mu,\h}^\text{pv}(\lambda,k)$ 
		and $\epsilon_{\mu,\h}^0(\lambda,k)$ are real. Hence by Formula~\eqref{eq:eps_zero}, if $\epsilon_{\mu,\h}(\lambda,k) = 0$, then $\epsilon_{\mu,\h}^0(\lambda,k) = 0$, which implies that $\lambda = 0$. But in this case $\epsilon_{\mu,\h}^\text{pv}(0,k) > 0$, which leads to a contradiction.
	\end{proof}


\subsection{An explicit Poisson equilibrium}\label{sec:explicit-Poisson-equilibrium}


	We detail here an example for which the quantum kernel can be computed explicitly. We define the so-called Poisson equilibrium
	\begin{equation}\label{Poisson-equilibrium}
		\mu_P(v):=\frac{1}{\pi^2(1+\sabs{v}^2)^2} \, ,
	\end{equation}
	which is a radial, strictly decreasing and normalized equilibrium whose Fourier transform is given by
	\begin{equation}\label{eq:mu_P_Fourier}
		\widehat{\mu}_P(\xi) = e^{-2\pi \n{\xi}} .
	\end{equation}

	\begin{prop}\label{prop:Green-explicit-solution}
		We consider the linearized Hartree equation~\eqref{eq:Hartree_linear} around the homogeneous state $\opmu=\mu_P(\opp)$. Then if $\rho_\h$ is the associated spatial density, its spatial Fourier transform satisfies
		\begin{equation}\label{eq:Green-explicit-solution}
			\widehat\rho_\h(t,k) = \widehat{\rho_\h^\circ}(t,k)- 2\pi\CK \int_0^t e^{-\n{k}(t-s)} \, \frac{\sin\!\lt(2\pi\Omega_k\lt(t-s\rt)\rt)}{\Omega_k} \, \widehat{\rho_\h^\circ}(s,k) \d s \, ,
		\end{equation}
		where
		\begin{equation*}
			\Omega_k=\sqrt{\CK+\tfrac{h^2\n{k}^4}{4}}
		\end{equation*}
		and
		\begin{equation*}
			\rho_\h^\circ(t,x) =\intd \Wigner{\op^\init}(x-tv,v) \d v
		\end{equation*}
		is the free-transport density.
	\end{prop}
	
	\begin{remark}
		One can write Equation~\eqref{eq:Green-explicit-solution} in the form $\widehat\rho = G_\textnormal{P} *_t \widehat{\rho^\circ}$. Then in the sense of distributions, if $\CK=(2\pi)^{-2}$,
		\begin{equation}\label{Green-classical-limit}
			G_\textnormal{P}(t,k) \longrightarrow
			\delta_0(t)- \indic_{\{t\ge0\}} \, \sin(t) \, e^{-2\pi \n{k}t} ,
		\end{equation}
		which is the Green kernel of the classical linearized Vlasov--Poisson system.
	\end{remark}

	\begin{remark}
		We emphasize that $\mu_P(v)\sim |v|^{-4}$ as $|v|\to\infty$ and
		\begin{equation*}
			\intd \n{v}^2 \mu_P(v) \d v = +\infty \, .
		\end{equation*}
		Thus, this subsection is about an exactly solvable model calculation; it is not an example covered by the finite-second-moment hypotheses used in the small-frequency expansion below.
	\end{remark}
	
	\begin{proof}[Proof of Proposition~\ref{prop:Green-explicit-solution}]
		Using the Duhamel formula on Equation~\eqref{eq:transfo} and then integrating in $v$ gives
		\begin{equation}\label{Volterra-rho-linear}
			\widehat\rho_\h(t,k) = \widehat{\rho_\h^\circ}(t,k) + \frac{2i\pi\CK\, k}{\n{k}^2}\cdot \int_0^t
			\widehat{D_{hk}\mu_P}(\lt(t-s\rt) k) \,\widehat\rho_\h(s,k) \d s \, .
		\end{equation}
		We have 
		\begin{equation*} 
		\widehat{D_{hk}\mu_P}(t\, k)
			=\frac{k}{h\n{k}^2}
			\left(e^{i\pi h\n{k}^2 t}-e^{-i\pi h\n{k}^2 t}\right)
			\widehat{\mu_P}(t\, k)
			=\frac{2 i k}{h\n{k}^2}
			\sin(\pi h\n{k}^2 t) \, e^{-2\pi\n{k}t} .
		\end{equation*}
		It follows that Equation~\eqref{Volterra-rho-linear} is the Volterra equation
		\begin{equation}\label{Volterra-rho-linear-explicit}
			\widehat\rho_\h(t,k) = \widehat{\rho_\h^\circ}(t,k) + \int_0^t A_\textnormal{P}(t-s,k) \, \widehat\rho_\h(s,k) \d s \, ,
		\end{equation}
		with kernel
		\begin{equation}\label{AP-explicit}
			A_\textnormal{P}(t,k) :=-\frac{4\pi\CK}{h\n{k}^2} \sin(\pi h\n{k}^2t) \, e^{-2\pi\n{k}t} \, .
		\end{equation}
		We now determine the corresponding Green kernel. We extend $\widehat{\rho_\h^\circ}$ by zero to $t<0$ and define $G_\textnormal{P}$ by
		\begin{equation}\label{Green-convolution}
			\widehat\rho_\h(\cdot,k) = G_\textnormal{P}(\cdot,k)*_t \widehat{\rho_\h^\circ} \, .
		\end{equation}
		Let 
		\begin{equation*}
			\L_t g(\lambda):=\int_0^\infty e^{-2\pi\lambda t}g(t) \d t \, ,
		\end{equation*}
		for $\Re \lambda$ sufficiently large, and let $b_k := h\n{k}^2/2$. Formula~\eqref{AP-explicit} gives
		\begin{equation*}
			\L_t A_\textnormal{P}(\lambda,k)
			=-\frac{\CK}{(\lambda+\n{k})^2+b_k^2} \, .
		\end{equation*}
		Hence the resolvent identity associated with Equation~\eqref{Volterra-rho-linear-explicit} implies
		\begin{equation*}
			\L_tG_\textnormal{P}(\lambda,k)
			=\frac{1}{1-\L_tA_\textnormal{P}(\lambda,k)} = 1-\frac{\CK}{(\lambda+\n{k})^2+\Omega_k^2}
		\end{equation*}
		with $\Omega_k:=\sqrt{\CK+b_k^2}$. Taking the inverse Laplace transform, we obtain the explicit formula
		\begin{equation*} 
			G_\textnormal{P}(t,k) =\delta_0(t) - 2\pi\CK \,\indic_{\{t\ge0\}} \, \frac{\sin(2\pi\Omega_k\, t)}{\Omega_k} \, e^{-2\pi \n{k}t},
		\end{equation*}
		where $\delta_0$ denotes the Dirac mass at $t=0$, and from which the result follows.
	\end{proof}


\subsection{Zeros of the dielectric function when \texorpdfstring{$\Re\lambda < 0$}{lambda has a negative real part}}\label{sec:dielsmall}


	\begin{proof}
		In the case $\Re{\lambda} < 0$, the dielectric function is given by 
		\begin{equation*}
			\epsilon_{\mu,\h}(\lambda,k) = \epsilon_{\mu,\h}^+(\lambda,k) + i\,\widehat{K}(k) \,\epsilon_{\mu,\h}^0(\lambda,k)
		\end{equation*}
		with
		\begin{align*}
			\epsilon_{\mu,\h}^+(\lambda,k) &= 1 + \intd \frac{\widehat{K}(k) \, D_{h k}\mu(v)}{\vartheta - v_1}\d v
			= 1 + \widehat{K}(k) \, \epsilon_{\mu,\h}^1(\lambda,k),
			\\
			\epsilon_{\mu,\h}^0(\lambda,k) &= 2\pi\int_{\R^2} D_{h k} \mu(\vartheta, v_{23}) \d v_{23}
		\end{align*}
		where $\vartheta = i\lambda/\n{k}$ satisfies $\Im{\vartheta}< 0$ and $v_{23} = (v_2,v_3)$. This can also be written
		\begin{equation}\label{eq:def_tildeps}
			\epsilon_{\mu,\h}(\lambda,k) = 1 + \widehat{K}(k) \,\tilde{\epsilon}_{\mu,\h}(\lambda,k) \quad \text{ with }\quad \tilde{\epsilon}_{\mu,\h} = \epsilon_{\mu,\h}^1 + i\,\epsilon_{\mu,\h}^0 \, .
		\end{equation} 
		The formulas are similar when $\vartheta$ is real. 
		
		The term $\epsilon_{\mu,\h}^0$ only vanishes at $\vartheta = 0$ and when $\vartheta\to\infty$. When $\vartheta=0$ and $\mu$ is radial and strictly decreasing, we already know that $\epsilon_{\mu,\h}^1 > 0$. As a consequence, $\epsilon_{\mu,\h}$ and $\tilde{\epsilon}_{\mu,\h}$ do not vanish for $\vartheta \in \lt[-R,R\rt] + i\lt[-\alpha,\alpha\rt]$ for arbitrarily large $R$ and small $\alpha$. More precisely, if $h\in [0,1]$ and $\n{k}\leq k_0$, we can take $\alpha$ and $R$ such that
		\begin{equation}\label{eq:bound_convexity_mu_2}
			\n{\tilde{\epsilon}_{\mu,\h}(\lambda,k)} \geq \frac{1}{2} \min\lt( 2\pi \int_{\R^2} \nu_{k_0}(0,v_{23}) \d v_{23}, 
			\intd \frac{\nu_{k_0}(v)}{v_1} \d v \rt) =: \tilde{\nu}(k_0)\, .
		\end{equation}
		Hence it remains to study $\epsilon_{\mu,\h}$ for large $\n{\vartheta}$. We follow standard computations (see for instance~\cite[p. 28]{schekochihin_lectures_2023}).

		First the term $\epsilon_{\mu,\h}^0$ decays like $\vartheta^{-N-1}$, since using the definition of $D_{h k}\mu$ and a first order Taylor estimate gives
		\begin{equation}\label{eq:boundsmall}
			\int_{\R^2} \n{D_{h k} \mu(\vartheta,v_{23})} \d v_{23} \leq \frac{C}{\weight{\vartheta-h k/2}^{N+1}} \Nrm{\weight{v}^{N+4}\nabla\mu}{L^\infty} .
		\end{equation}
		We now look at the term $\epsilon_{\mu,\h}^+$ in which we perform a Taylor expansion in $\vartheta$. More precisely,
		\begin{equation*}
			\intd \frac{D_{h k}\mu(v)}{\vartheta - v_1} \d v = \intd \frac{1}{\vartheta} \sum_{j=0}^N D_{h k}\mu(v) \lt(\frac{v_1}{\vartheta}\rt)^j +\frac{D_{h k}\mu(v)}{\vartheta - v_1} \lt(\frac{v_1}{\vartheta}\rt)^{N+1} \d v \, .
		\end{equation*}
		Let us bound the last term, that we will call $I_R$. We have
		\begin{equation}\label{eq:I_R}
			I_R
			= \frac{1}{\vartheta^{N+1}} \int_{\n{\vartheta - v_1}\ge 1} \frac{F_{h k}(v)}{\vartheta - v_1} \d v + \frac{1}{\vartheta^{N+1}} \int_{\n{\vartheta - v_1} \le 1} \frac{F_{h k}(v)}{\vartheta - v_1} \d v
		\end{equation}
		where
		\begin{equation*}
			F_{h k}(v) = v_1^{N+1} D_{h k} \mu(v) \, .
		\end{equation*}
		The first integral is bounded uniformly in $h\in[0,1]$ since 
		\begin{equation*}
			\Nrm{F_{h k}}{L^1} \leq C_{k_0}\sNrm{\weight{v}^{N+1}\Dv\mu}{L^1}\,.
		\end{equation*}
		Since the function $v_1\mapsto \vartheta-v_1$ is odd on $\set{v_1\in\R:\n{\vartheta-v_1}\leq 1}$, the second integral equals
		\begin{equation*}
			\int_{\n{\vartheta - v_1} \le 1} \frac{F_{h k}(v)}{\vartheta - v_1} \d v = \frac{1}{\vartheta^{N+1}} \int_{\n{\vartheta - v_1} \le 1} \frac{F_{h k}(v) - F_{h k}(\vartheta,v_{23})}{\vartheta - v_1} \d v \, .
		\end{equation*}
		Now, $v_2$ and $v_3$ being fixed,
		\begin{equation*}
			\int_{\n{\vartheta - v_1} \le 1} \frac{F_{h k}(v) - F_{h k}(\vartheta,v_{23})}{\vartheta - v_1} \d v \le C \sup_{\n{\vartheta - v_1} \le 1} \n{\partial_{v_1} F_{h k}(v)} .
		\end{equation*}
		After integration over $v_{23}$, we deduce that the second integral in Equation~\eqref{eq:I_R} is bounded uniformly in $\vartheta$ and $\h$. As a consequence, $\n{I_R} \leq C_{k_0,\mu}\,\vartheta^{-N-1}$, uniformly in $\h$ and $k$.
	
		For the other terms, one can compute that
		\begin{align*}
			I_j := -\intd D_{h k}\mu(v)\, v_1^j\d v = \frac{1}{hk}\intd \mu(v)\( \(v_1+hk/2\)^j - \(v_1-hk/2\)^j\)\d v
		\end{align*}
		and so, in particular, $I_j = 0$ if $j$ is even by symmetry of $\mu$, and, with the notation 
		\begin{equation*}
		M_n = \intd \mu(v) \, v_1^n \d v \, ,
		\end{equation*}
		one obtains
		\begin{align*}
			I_1 &= M_0\, , \qquad I_3 = 3\, M_2 + \lt(hk/2\rt)^2 M_0 \, ,
			\\
			I_5 &= 5 \, M_4 + \frac{5}{2} \(hk\)^2 M_2 + \lt(hk/2\rt)^4 M_0 \, .
		\end{align*}
		Hence doing the expansion up to $N = 6$ gives
		\begin{equation}\label{eq:expansion_tildeps}
			\tilde\epsilon_{\mu,\h}(\lambda,k) 
			= -\(\frac{M_0}{\vartheta^2} + \frac{3\,M_2}{\vartheta^4} + \frac{M_0\,k^2 h^2}{3\,\vartheta^4} 
			+ \frac{5\,M_4 + \eps(\vartheta)}{\vartheta^6}\)
		\end{equation}
		with
		\begin{equation*}
			\eps(\vartheta) = \(10\, M_2 + \(\tfrac{hk}{2}\)^2 M_0\) \(\tfrac{hk}{2}\)^2 + \tfrac{I_7}{\vartheta}\,.
		\end{equation*}
		That is
		\begin{equation*}
			\epsilon_{\mu,\h}(\lambda,k) = 1 - \widehat{K}(k) \(\frac{M_0}{\vartheta^2} + \frac{3\,M_2}{\vartheta^4} + \frac{M_0\,k^2h^2}{3\,\vartheta^4} + \frac{5\,M_4 + \eps(\vartheta)}{\vartheta^6}\).
		\end{equation*}
		The same arguments lead to a similar expansion for $\partial_\lambda \epsilon_{\mu,\h}(\lambda,k)$.
		
		Finally, the relation $\epsilon_{\mu,\h}(\lambda,k) = 0$ becomes
		\begin{equation*}
			\frac{M_0}{\vartheta^2} + \frac{3\,M_2}{\vartheta^4} + \frac{M_0\,k^2 h^2}{3\,\vartheta^4} + \frac{5\,M_4 + \eps(\vartheta)}{\vartheta^6}+ O(\n{\vartheta}^{-7}) = \widehat{K}(k)^{-1}
		\end{equation*}
		that is
		\begin{equation*}
			-\frac{M_0\,k^2}{\lambda^2} + \frac{3\,M_2k^4}{\lambda^4} + \frac{M_0\,k^6h^2}{3\,\lambda^4} - \frac{5\,M_4\,k^6 + \eps(\vartheta)\,k^6}{\lambda^6}+ O(\n{\vartheta}^{-7}) = \widehat{K}(k)^{-1}.
		\end{equation*}
		We are now ready to investigate the existence of roots of the dielectric function.
	
		First, for large $k$, we observe that $\tilde\epsilon_{\mu,\h}$ is bounded for $\vartheta\in \R + i \lt[-\alpha, \alpha\rt]$, uniformly in $\h$. Thus, provided $\n{k}$ is large enough, $\widehat{K}(k) \n{\tilde\epsilon_{\mu,\h}} \leq 1/2$. As a consequence, there is no eigenvalue $\lambda$ such that $\Re \lambda > - \alpha \n{k}$.
	
		Let us now consider the case of small $k$. Since $\widehat{K}(k) = \CK \n{k}^{-2}$, we have
		\begin{equation*}
			-\frac{M_0}{\lambda^2} + \frac{3\,M_2\n{k}^2}{\lambda^4} + \frac{M_0\n{k}^4 h^2}{3\,\lambda^4} - \frac{5\,M_4\n{k}^4}{\lambda^6}+ O(\n{k}^5) = \frac{1}{\CK} \, .
		\end{equation*}
		To order $\n{k}^2$, since $\Re{\lambda}<0$ this gives
		\begin{equation*}
			\frac{1}{\lambda^2} = \frac{M_0}{6\,M_2\n{k}^2} \Big(1 - \sqrt{1+\tfrac{12\,M_2\n{k}^2}{\CK\,M_0^2}}\Big) + O(\n{k}^2) = - \frac{1}{\CK\,M_0} + O(\n{k}^2)
		\end{equation*}
		so in particular, $\vartheta$ is large if and only if $k$ is small. 
		We now substitute this expansion into the remaining terms of the
		dispersion relation. It gives
		\begin{equation*}
			\frac{1}{\lambda^4}
			=\frac{1}{\CK^2M_0^2}
			-\frac{6M_2}{\CK^3M_0^4}\n{k}^2
			+O(\n{k}^4),
			\qquad
			\frac{1}{\lambda^6}
			=-\frac{1}{\CK^3M_0^3}+O(\n{k}^2).
		\end{equation*}
		On the other hand, multiplying the preceding dispersion relation by
		$\CK$ and isolating its first term gives
		\begin{equation*}
			\frac{\CK M_0}{\lambda^2}
			=-1
			+\frac{3\CK M_2\n{k}^2}{\lambda^4}
			+\frac{\CK M_0h^2\n{k}^4}{3\lambda^4}
			-\frac{5\CK M_4\n{k}^4}{\lambda^6}
			+O(\n{k}^5).
		\end{equation*}
		Applying the preceding expansions of $\lambda^{-4}$ and
		$\lambda^{-6}$ to the right-hand side, we obtain
		\begin{equation*}
			\frac{\CK M_0}{\lambda^2}
			=-1+\frac{3M_2}{\CK M_0^2}\n{k}^2
			+	{\left(
			-\frac{18M_2^2}{\CK^2M_0^4}
			+\frac{5M_4}{\CK^2M_0^3}
			+\frac{h^2}{3\CK M_0}
			\right)\n{k}^4}+O(\n{k}^5).
		\end{equation*}
		
		This gives
		\begin{equation*}
		\lambda_{\h,\pm} = \pm i \CK^{1/2} M_0^{1/2}
		\pm \frac{3 i M_2}{2 \CK^{1/2} M_0^{3/2}} \n{k}^2 + O(\n{k}^4) \, .
		\end{equation*}
		Moreover, observe that when doing the difference $\lambda_{\h,\pm} - \lambda_{0,\pm}$, all the terms of the expansion without the Planck constant cancel, and the remaining terms will always contain a factor $h^2$. It gives
		\begin{equation*}
			\n{\lambda_{\h,\pm} - \lambda_{0,\pm}} \lesssim \n{k}^4 h^2 .
		\end{equation*}
		Similar computations lead to
		\begin{equation*}
			\partial_k \lambda_{\h,\pm} = \pm \frac{3 i\, M_2}{ \CK^{1/2} M_0^{3/2}}\, k + O(\n{k}^3) \, 
		\end{equation*}
		and
		\begin{equation*}
			\n{\partial_k \lambda_{\h,\pm} - \partial_k \lambda_{0,\pm}} \lesssim \n{k}^3 h^2 .
		\end{equation*}
		We now turn to the description of the real part of $\lambda_{\h,\pm}$, which turns out to be much smaller than its real part.
		We first solve 
		\begin{equation} \label{eq:approximate}
			\epsilon_{\mu,\h}^+(\lambda,k) = 0 \, .
		\end{equation}
		Using Inequality~\eqref{eq:boundsmall}, we see that $\epsilon_{\mu,\h}$ and $\epsilon_{\mu,\h}^+$ have the same asymptotic expansions. As a consequence, Equation~\eqref{eq:approximate} has two solutions $\lambda^{\textnormal{app}}_\pm$ such that $\lambda^{\textnormal{app}}_\pm \sim i \CK^{1/2} M_0^{1/2}$. In view of Inequality~\eqref{eq:boundsmall}, we consider $\epsilon_{\mu,\h}^0$ as a perturbation and apply the implicit function theorem to obtain
		\begin{equation*}
			\re \lambda_{\h,\pm} \sim - i \widehat{K}(k) \, \frac{\epsilon_{\mu,\h}^0(\lambda^{\textnormal{app}}_\pm,k)}{\partial_\lambda \epsilon_{\mu,\h}^+(\lambda^{\textnormal{app}}_\pm,k)} \, ,
		\end{equation*}
		which gives the asymptotic expansion of $\re \lambda_{\h,\pm}$.
	
		Let us now turn to medium $k$. We note that $\epsilon_{\mu,\h} \to 1$ as $\n{\lambda} \to \infty$. As $\epsilon_{\mu,\h}$ does not vanish on the imaginary axis, this means that possible eigenvalues stay away from the imaginary axis, and have a real part smaller than $- \sigma$ for some positive and fixed $\sigma$.
	\end{proof}


\section{Study of the Green functions}



\subsection{Green functions in Fourier}


	We rewrite Equation~\eqref{division} under the form 
	\begin{equation*}
		\widetilde{\rho}(\lambda,k) = \widetilde{\rho^\circ}(\lambda,k) + \Bigl( \frac{1}{\epsilon_{\mu,\h}(\lambda,k)} - 1 \Bigr) \,\widetilde{\rho^\circ}(\lambda,k) \, .
	\end{equation*}
	For a contour $\Gamma\subseteq \CC$ to be precised later, we introduce $\widehat{G}_\h$, the inverse Laplace transform of $\epsilon_{\mu,\h}^{-1}-1$, given by
	\begin{equation}\label{eq:def_Gk}
		\widehat{G}_\h(t,k) = \frac{1}{i} \int_\Gamma \frac{1-\epsilon_{\mu,\h}(\lambda,k)}{\epsilon_{\mu,\h}(\lambda,k)} \, e^{2\pi\lambda t} \d\lambda \, ,
	\end{equation}
	in such a way that
	\begin{equation}\label{eq:inversion2}
		\widehat{\rho}(t,k) = \widehat{\rho^\circ}(t,k) + \int_0^t \widehat{G}_\h(t - \tau,k) \,\widehat{\rho^\circ}(\tau,k) \d\tau\, .
	\end{equation}
	For small $\n{k}$, we will take a contour on the left of the two roots $\lambda_\pm(k)$, and for medium of large $\n{k}$, a contour on the right of all the (possible) roots of $\eps_\mu$. To ensure a smooth transition, we choose a function $\phi(k)$ of small compact support, which vanishes near $k = 0$ and write $\hatG_\h = \phi \,\hatG_\h + \lt(1 - \phi\rt) \hatG_\h$. We use a contour on the left of the roots $\lambda_\pm(k)$ for $\phi \,\hatG_\h$ and on the right of these roots for $\lt(1 - \phi\rt) \hatG_\h$.
 
	Recall $\vartheta = i\lambda/|k|$. For $\n{\vartheta} \le R$, namely for $\n{\lambda} \le R \n{k}$ we follow the contour $\vartheta \in \lt[-R,R\rt] - i\, \alpha$, namely $\lambda \in \Gamma_1$ with
	\begin{equation*}
		\Gamma_1 = [\lambda_2,\lambda_3] \quad \text{ with } \quad\lambda_2 = - i \,R \n{k} - \alpha \n{k} \quad \text{ and } \quad \lambda_3 = i \,R \n{k} - \alpha \n{k} .
	\end{equation*}
	For $\n{\vartheta} \geq R$, we choose the contour
	\begin{equation*}
		\Gamma_2 = \lambda_2 - e^{i \beta} \,\R_+ \quad \text{ and } \quad \Gamma_3 = \lambda_3 - e^{- i \beta}\, \R_+
	\end{equation*}
	for some $\beta \in (0,\pi/2)$ close to $\pi / 2$. We choose $\beta$ to avoid the poles $ \lambda_{\h,\pm}(k)$ of $\epsilon_{\mu,\h}^{-1}$. We recall that $\lambda_{\h,\pm}(k) \sim \pm i\sqrt{\CK\,M_0}$ with $\Re{\lambda_\pm(k)} \sim {}(k)$. 
	
	If $k$ is small and $\phi(k) = 1$, we choose $\alpha$ such that these two poles are on the right of the contour $\Gamma = \Gamma_2 \cup \Gamma_1 \cup \Gamma_3$, which leads to the residues
	\begin{equation*}
		\hatGosc_{\h,\pm}(t,k) = \mathrm{Res}_{\lambda_{\h,\pm}(k)}\!\lt( \epsilon_{\mu,\h}^{-1}(\lambda,k)\rt) e^{2\pi\lambda_{\h,\pm}(k) t}
		= a_{\h,\pm}(k)\, e^{2\pi\lambda_{\h,\pm}(k) t} \, . 
	\end{equation*}%
	That is, we have more precisely
	\begin{equation}\label{eq:hatG_expansion}
		\hatG_\h = \hatGosc_{\h,+} + \hatGosc_{\h,-} + \hatGr_\h
	\end{equation}
	with
	\begin{equation}\label{eq:def_hatGr}
		\hatGr_\h(t,k) = \frac{1}{i} \int_\Gamma \frac{1-\epsilon_{\mu,\h}(\lambda,k)}{\epsilon_{\mu,\h}(\lambda,k)} 
		\, e^{2\pi\lambda t} \d\lambda \, .
	\end{equation}
	We note that in the integral
	\begin{equation*}
		\int_\Gamma e^{2 \pi \lambda t} \d\lambda
	\end{equation*}
	we can move the contour $\Gamma$ towards the left. As a consequence, this integral decays faster than any exponential. It remains thus to bound
	\begin{equation*}
		\widehat{H}_\h^\textnormal{r}(t,k) = \frac{1}{i} \int_\Gamma \frac{e^{2\pi\lambda t}}{\epsilon_{\mu,\h}(\lambda,k)} \d\lambda \, .
	\end{equation*}
	If $k$ is small and $\phi(k) \ne 1$, we use two contours, one on the left of the roots to bound $\phi \, \hatG_\h$ 
	and one on their right to bound $\lt(1 - \phi\rt) \hatG_\h$.
	
	For medium or large $k$, we choose $\alpha$ such that the contour is on the right of all the poles of $\eps_\mu$ (if such poles still exist, since we have only proved their existence for small $k$). In this case, we simply have $\hatG_\h = \hatGr_\h$.

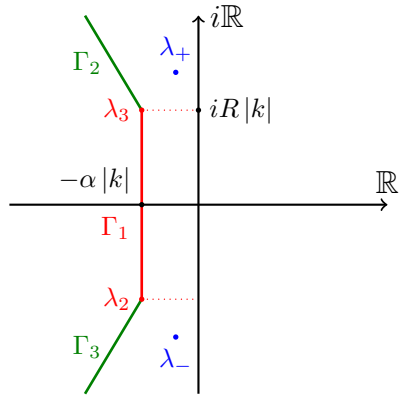
\medskip
\begin{center}
	\begin{tikzpicture}[scale=0.25]
		
		\draw[->,line width = 0.8pt](-10,0)--(10,0) node[above]{$\R$};
		
		\draw[->,line width = 0.8pt](0,-10)--(0,10) node[right]{$i\R$};
		
		

		
		\draw [line width = 1, red] (-3,-5)--(-3,5) node[midway, below left, color=red, align=center]{\footnotesize $\Gamma_1$};
		\draw [line width = 1, vertFonce] (-3,5)--(-6,10) node[midway, left, color=vertFonce, align=center]{\footnotesize $\Gamma_2$}; 
		\draw [line width = 1, vertFonce] (-3,-5)--(-6,-10) node[midway, left, color=vertFonce, align=center]{\footnotesize $\Gamma_3$}; 
		
		\foreach \x in {1, ..., 9}
			\fill[color=red] (-3*\x/10,-5) circle [radius=1pt];
		\foreach \x in {1, ..., 9}
			\fill[color=red] (-3*\x/10,5) circle [radius=1pt];


		\fill[color=black] (0,5)  circle[radius=4pt] node[right]{\footnotesize$iR\n{k}$};
		\fill[color=black] (-3,0)  circle[radius=4pt] node[above left]{\footnotesize$-\alpha\n{k}$};
		
		\fill[color=red] (-3,-5) circle[radius=4pt] node[left]{\footnotesize$\lambda_2$};
		\fill[color=red] (-3,5)  circle[radius=4pt] node[left]{\footnotesize$\lambda_3$};

		\fill[color=blue] (-1.2,7)  circle[radius=4pt] node[above]{\footnotesize$\lambda_+$};
		\fill[color=blue] (-1.2,-7)  circle[radius=4pt] node[below]{\footnotesize$\lambda_-$};
	\end{tikzpicture}
\end{center}
\medskip

	To bound $\Gosc_{\h,\pm}$ we need to study $a_{\h,\pm}(k)$.
	\begin{lem}
		We have, as $k \to 0$,
		\begin{equation}\label{eq:a_expression}
			a_{\h,\pm}(k) =\pm i\pi\sqrt{\CK M_0} \mp\frac{3i\pi M_{2,\parallel}}{2\sqrt{\CK}\,M_0^{3/2}}\n{k}^2 + O(\n{k}^4) \, ,
		\end{equation}
		\begin{equation*}
			\n{a_{\h,\pm}(k) - a_{0,\pm}(k)} \lesssim \n{k}^4 h^2.
		\end{equation*}
	\end{lem}

	\begin{proof}
		By definition,
		\begin{equation*}
			a_{\h,\pm}(k) = \bigl( \partial_\lambda \epsilon_{\mu,\h}(\lambda_{\h,\pm}(k),k) \bigr)^{-1},
		\end{equation*}
		which, after straightforward computations, leads to Equation~\eqref{eq:a_expression} and to the estimates on $a_{h,\pm} - a_{0,\pm}$.
	\end{proof}
	
	We now bound $\hatGr_\h$.
	
	\begin{lem}\label{lem:estim_hatGr}
		For any $\h\in[0,1]$, it holds
		\begin{equation}\label{eq:hatGr1}
			\sabs{\hatGr_\h(t,k)} \leq C \, \tfrac{\n{k}^3}{\weight{k}^4} \, e^{-\sigma\n{k}t}
		\end{equation}
		for some $C,\sigma>0$ independent of $k$ and $\h$. Moreover,
		\begin{equation}\label{eq:splitr}
			\n{\hatGr_\h(t,k) - \hatGr_0(t,k)} \le C \, \tfrac{\n{k}^3}{\weight{k}^4} \, e^{-\sigma\n{k}t} \min(1, h^2 |k|^2) \, .
		\end{equation}
		In addition, for $t>0$, $k\neq0$, and every multi-index $\n{\alpha}\leq2$,
		\begin{align}\label{eq:Gr-symbol-derivatives}
			\n{\partial_k^\alpha\hatGr_\h(t,k)}
			&\leq C \, \frac{\n{k}^{3-\n{\alpha}}}{\weight{k}^4}
			\weight{\n{k}t}^{\n{\alpha}}e^{-\sigma\n{k}t},
			\\\label{eq:Gr-difference-symbol-derivatives}
			\n{\partial_k^\alpha
			\bigl(\hatGr_\h-\hatGr_0\bigr)(t,k)}
			&\leq C \, h^2\, \frac{\n{k}^{5-\n{\alpha}}}{\weight{k}^4}
			\weight{\n{k}t}^{\n{\alpha}} e^{-\sigma\n{k}t} .
		\end{align}
	\end{lem}

	\begin{proof}[Proof of Inequality~\eqref{eq:hatGr1}]
		The estimates cannot be obtained by taking absolute values of the integral of $\epsilon_{\mu,\h}^{-1}$ on $\Gamma$: this would miss a cancellation between the integral and the poles. Instead we have to use a particular decomposition of $\epsilon_{\mu,\h}^{-1}$. For the classical setting, we refer to~\cite[Lemma~2.7]{bedrossian_linearized_2022} (low frequencies) and~\cite[Lemma~2.8]{bedrossian_linearized_2022} (non-small frequencies). See also~\cite[Section~5.2.3]{han-kwan_linearized_2021} and~\cite[Theorem~1.2~(ii), especially~(1.25)]{ionescu_stability_2023}.

		Let $r := \n{k}$. We first study the case $0<r\leq r_0$ where $r_0$ is small enough. 
		The expansion~\eqref{eq:expansion_tildeps} leads to 
		\begin{equation}\label{eq:Lh-ray-expansion}
			1-\epsilon_{\mu,\h}(\lambda,k) = - \frac{\CK M_0}{\lambda^2} +\frac{b_\h(r)}{\lambda^4} + \frac{r^4}{\lambda^6} \, \zeta_{\h,r}(\lambda) \, ,
		\end{equation}
		where
		\begin{equation*}
			b_\h(r) := 3\CK M_{2,\parallel} \, r^2 +\frac{\CK M_0}{4} \, h^2r^4 \, .
		\end{equation*}
		The functions $\zeta_{\h,r}$ are holomorphic in the conic region~\eqref{eq:def_conic}, uniformly bounded on $\Gamma_2$ and $\Gamma_3$, and decay at infinity. 
		We note that
		\begin{equation}\label{eq:zeta-quantum-classical-difference}
			\n{\zeta_{\h,r}(\lambda)-\zeta_{0,r}(\lambda)} \leq C\, h^2r^2 .
		\end{equation}
		The same calculation may be differentiated at fixed $z=\lambda/r$. Taking into account
		the decay along $\Gamma_2$ and $\Gamma_3$, we have, for $0\leq j\leq2$
		\begin{align}\label{eq:zeta-radial-derivatives}
			\n{\partial_r^j\zeta_{\h,r}(rz)} &\leq C_j \, r^{-j} \weight{z}^{-1} ,
			\\\label{eq:zeta-difference-radial-derivatives}
			\n{\partial_r^j\bigl(\zeta_{\h,r}-\zeta_{0,r}\bigr)(rz)} &\leq C_j \, h^2 r^{2-j} \weight{z}^{-1} .
		\end{align}
		Following~\cite[Lemma~2.7]{bedrossian_linearized_2022}, we introduce the rational function
		\begin{equation}\label{eq:rational-surgery-correction}
			P_{\h,r}(\lambda) :=\frac{\CK M_0}{\lambda^2+\CK M_0} - \frac{b_\h(r)}{(\lambda^2+\CK M_0)^2}
		\end{equation}
		and the new integrand
		\begin{equation}\label{eq:def-corrected-post-residue-integrand}
			Q_{\h,r}(\lambda) := \frac{1-\epsilon_{\mu,\h}(\lambda,k)}{\epsilon_{\mu,\h}(\lambda,k)} + P_{\h,r}(\lambda) \, .
		\end{equation}
		The only poles of $P_{\h,r}$ are $\lambda = \pm i\sqrt{\CK M_0}$, which lie on the right of $\Gamma$. Moreover $P_{\h,r}(\lambda) = O(\n{\lambda}^{-2})$ as $\n{\lambda}\to\infty$. Therefore, for $t>0$, closing the contour in the left half-plane gives
		\begin{equation}\label{eq:rational-correction-zero-integral}
			\int_\Gamma e^{2\pi\lambda t}P_{\h,r}(\lambda) \d\lambda=0 \, .
		\end{equation}
		Therefore, it follows from equations~\eqref{eq:def_hatGr} and~\eqref{eq:def-corrected-post-residue-integrand} that
		\begin{equation}\label{eq:Gr-corrected-contour}
			\hatGr_\h(t,k) =\frac1i\int_\Gamma e^{2\pi\lambda t} \, Q_{\h,r}(\lambda) \d\lambda \, .
		\end{equation}

		We now estimate $Q_{\h,r}$. On $\Gamma_1$, $\n{\lambda}\leq C\, r$. By~\eqref{eq:def_tildeps} and~\eqref{eq:bound_convexity_mu_2}, provided $r_0$ is small enough,
		\begin{equation*}
			\n{\epsilon_{\mu,\h}(\lambda,k)}
			\geq \CK \, r^{-2} \n{\widetilde\epsilon_{\mu,\h}(i\lambda/r,r)} - 1
			\gtrsim r^{-2}.
		\end{equation*}
		Hence
		\begin{equation*}
			\frac{1-\epsilon_{\mu,\h}(\lambda,k)}{\epsilon_{\mu,\h}(\lambda,k)} + 1 = \frac{1}{\epsilon_{\mu,\h}(\lambda,k)} = O(r^2) \, .
		\end{equation*}
		On the other hand,
		\begin{equation*}
			P_{\h,r}-1 = - \frac{\lambda^2}{\lambda^2+\CK M_0} - \frac{b_\h(r)}{(\lambda^2+\CK M_0)^2} =O(r^2) \, .
		\end{equation*}
		Adding the last two identities proves the existence of a constant $C$ independent of $k$ and $\h$ such that for any $\lambda\in\Gamma_1$,
		\begin{equation}\label{eq:post-residue-central-bound}
			\n{Q_{\h,r}(\lambda)}\leq C\, r^2\, .
		\end{equation}

		On $\Gamma_2$ and $\Gamma_3$, we compute $Q_{h,r}$ explicitly. Recalling Equation~\eqref{eq:rational-surgery-correction}, we have
		\begin{align*}
			Q_{\h,r}(\lambda)
			&=\frac{1-\epsilon_{\mu,\h}}{\epsilon_{\mu,\h}}
			+\frac{\CK M_0}{\lambda^2+\CK M_0}
			-\frac{b_\h(r)}{(\lambda^2+\CK M_0)^2}
			=\frac{N_{\h,r}(\lambda)}{D_{\h,r}(\lambda)} \, ,
		\end{align*}
		where $\epsilon_{\mu,\h} = \epsilon_{\mu,\h}(\lambda,k)$ and
		\begin{align}\label{eq:Dh-exact-definition}
			D_{\h,r}(\lambda) &:= \epsilon_{\mu,\h} \lt(\lambda^2+\CK M_0\rt)^2,
			\\\notag
			N_{\h,r}(\lambda) &:= \lt(1-\epsilon_{\mu,\h}\rt) \lt(\lambda^2+\CK M_0\rt)^2 +\epsilon_{\mu,\h} \lt(\CK M_0(\lambda^2+\CK M_0)-b_\h(r)\rt)
			\\\label{eq:Nh-exact-definition}
			&= \lt(\lambda^2\lt(1-\epsilon_{\mu,\h}\rt)+\CK M_0\rt) \lt(\lambda^2+\CK M_0\rt) - \epsilon_{\mu,\h} \, b_\h(r) \, .
		\end{align}
		We first study $N_{\h,r}$. Substituting expansion~\eqref{eq:Lh-ray-expansion} into~\eqref{eq:Nh-exact-definition} gives
		\begin{align*}
			N_{\h,r}
			&=\left(\frac{b_\h}{\lambda^2}
			+\frac{r^4\zeta_{\h,r}}{\lambda^4}\right)
			(\lambda^2+\CK M_0)
			-b_\h\left(1+\frac{\CK M_0}{\lambda^2}
			-\frac{b_\h}{\lambda^4}
			-\frac{r^4\zeta_{\h,r}}{\lambda^6}\right)
			\\
			&=\frac{r^4\, \zeta_{\h,r}}{\lambda^2}
			+\frac{\CK M_0\, r^4\, \zeta_{\h,r}+b_\h(r)^2}{\lambda^4}
			+\frac{b_\h(r)\, r^4\, \zeta_{\h,r}}{\lambda^6} \, .
		\end{align*}
		Uniformly for $h\in[0,1]$ and $0<r\leq r_0\leq 1$, we have $\n{b_\h(r)}\leq C\, r^2$ and $\n{\zeta_{\h,r}(\lambda)}\leq C$, which leads to
		\begin{equation}\label{eq:Nh-explicit-ray-bound}
			\n{N_{\h,r}(\lambda)} \leq
			\begin{cases}
				C_R\, r^4 \n{\lambda}^{-4}, & \text{ if } R\, r \leq \n{\lambda}\leq 1 \, ,
				\\
				C\, r^4 \n{\lambda}^{-2}, & \text{ if } \n{\lambda} \geq 1 \, ,
			\end{cases}
		\end{equation}
		with $C_R = C \lt(1+R^{-2}\rt)$.

		It remains to study $D_{\h,r}$. By Equation~\eqref{eq:Lh-ray-expansion}, we can write, on $\Gamma_2$ and $\Gamma_3$,
		\begin{equation*}
			\epsilon_{\mu,\h}(\lambda,k) = \frac{\lambda^2+\CK M_0}{\lambda^2} + \mathcal E_{\h,r}(\lambda) \, ,
		\end{equation*}
		where $\mathcal E_{\h,r}$ is a function of $\lambda$ which satisfies, since $\n{b_\h(r)}\leq C \, r^2$ and $\n{\zeta_{\h,r}(\lambda)}\leq C$,
		\begin{equation*}
			\n{\mathcal E_{\h,r}(\lambda)}
			\leq C\left(\frac{r^2}{\n{\lambda}^4} + \frac{r^4}{\n{\lambda}^6}\right).
		\end{equation*}
		If $R \, r\leq \n{\lambda} \leq 1$, the rays stay uniformly away from $\pm i\sqrt{\CK M_0}$, and hence $\n{\lambda^2+\CK M_0}\geq c_0$ for some constant $c_0 > 0$. Therefore,
		\begin{equation*}
			\frac{\n{\mathcal E_{\h,r}(\lambda)}}{\n{\lambda^2+\CK M_0}/\n{\lambda}^2} \leq C\lt(R^{-2}+R^{-4}\rt).
		\end{equation*}
		If $\n{\lambda}\geq1$, the geometry of the rays instead gives $\n{\lambda^2+\CK M_0}\geq c_0\n{\lambda}^2$, while $\n{\mathcal E_{\h,r}(\lambda)}\leq C\lt(r^2+r^4\rt)$. Choosing first $R$ sufficiently large and then $r_0$ sufficiently small, the reverse triangle inequality consequently yields
		\begin{equation}\label{eq:ray-denominator-comparison}
			\n{\epsilon_{\mu,\h}} \geq \frac{\n{\lambda^2+\CK M_0}}{2 \n{\lambda}^2} \, .
		\end{equation}
		By Equation~\eqref{eq:Dh-exact-definition}, this implies
		\begin{equation}\label{eq:Q-denominator-explicit-lower-bound}
			\n{D_{\h,r}(\lambda)} \geq \frac{\n{\lambda^2+\CK M_0}^3}{2 \n{\lambda}^2} \, .
		\end{equation}
		On $\Gamma_2$ and $\Gamma_3$, which stay away from $\pm i\sqrt{\CK M_0}$, one has $\n{\lambda^2+\CK M_0}\geq c$ for $\n{\lambda}\leq1$ and $\n{\lambda^2+\CK M_0}\geq c\n{\lambda}^2$ for $\n{\lambda}\geq 1$. Dividing~\eqref{eq:Nh-explicit-ray-bound} by~\eqref{eq:Q-denominator-explicit-lower-bound} therefore gives bounds of order, respectively, $r^4\n{\lambda}^{-2}$ and $r^4\n{\lambda}^{-6}$. In particular, for any $\lambda\in\Gamma_2\cup\Gamma_3$ such that $\n{\lambda}\geq R \, r$
		\begin{equation}\label{eq:post-residue-ray-bound}
			\n{Q_{\h,r}(\lambda)} \leq C_R \, \frac{r^4}{\n{\lambda}^2} \, .
		\end{equation}
		All constants above are uniform in $\h$. Indeed, recall that $F$ denotes the holomorphic extension of the one-dimensional velocity marginal~\eqref{eq:marginal}. Its centered difference with step $hr$ is
		\begin{equation*}
			D_{hr}F(z) := \frac{F(z+hr/2)-F(z-hr/2)}{hr} =\frac12\int_{-1}^{1}F'(z+hrs/2)\d s,
		\end{equation*}
		with the convention $D_0F=F'$. Thus $D_{hr}F$ inherits uniformly the holomorphy and thin-tail bounds of $F'$. This proves that the constants in~\eqref{eq:Lh-ray-expansion}--\eqref{eq:post-residue-ray-bound} are independent of $h\in[0,1]$.

		Since $\Gamma_1$ has length $O(r)$ and $\re\lambda\leq-\sigma\, r$ on the full contour, we obtain for any $r\in(0,r_0]$,
		\begin{equation*}
			\n{\hatGr_\h(t,k)} \leq C \, e^{-\sigma rt} \lt(r\,r^2 + \int_{Rr}^{\infty}\frac{r^4}{s^2} \d s\rt) \leq C \, r^3 \, e^{-\sigma rt}.
		\end{equation*}
		We next treat non-small frequencies as in~\cite[Lemma~2.8]{bedrossian_linearized_2022}. If $r\geq r_0^{-1}$, there are no poles in $\Re\lambda\geq-\sigma\, r$, and we deform to the vertical line
		\begin{equation*}
			\lambda=-\sigma\, r+i\,\omega\, , \qquad \text{ with } \omega\in\R \, .
		\end{equation*}
		The analytic decay of $D_{hr}F$, the factor $\widehat K(k)=\CK\, r^{-2}$, and the uniform denominator bound on this line give
		\begin{equation}\label{eq:non-small-resolvent-integrand}
			\n{\frac{1-\epsilon_{\mu,\h}(-\sigma r+i\omega)}{\epsilon_{\mu,\h}(-\sigma r+i\omega)}} \leq \frac{C}{r^2+\omega^2} \, .
		\end{equation}
		Therefore for any $r\geq r_0^{-1}$,
		\begin{equation*}
			\n{\hatGr_\h(t,k)} \leq C\, e^{-\sigma rt} \int_\R \frac{\d\omega}{r^2+\omega^2} = C \, r^{-1} e^{-\sigma rt},
		\end{equation*}
		On the compact annulus $r_0\leq r\leq r_0^{-1}$, uniform stability gives the same denominator bound and analytic decay gives an integrable majorant in $\omega$. Hence
		\begin{equation*}
			\n{\hatGr_\h(t,k)}\leq C \, e^{-\sigma rt}
		\end{equation*}
		there. Since $r^3/\weight{r}^4$ is comparable to $r^3$ for $r\leq r_0$, to $r^{-1}$ for $r\geq r_0^{-1}$, and to a positive constant on the intervening annulus, the three estimates prove Inequality~\eqref{eq:hatGr1}.
	\end{proof}

	\begin{proof}[Proof of inequalities~\eqref{eq:splitr}, \eqref{eq:Gr-symbol-derivatives} and~\eqref{eq:Gr-difference-symbol-derivatives}]
		Again, we write $r:=\n{k}$ and apply the same construction to the difference. The centered-difference identity gives, for every $0\leq j\leq2$,
		\begin{equation}\label{eq:centered-difference-endpoints-Gr}
			\n{\partial_z^j\bigl(D_{hr}F(z)-F'(z)\bigr)} \leq C\min\{1,h^2r^2\} \, .
		\end{equation}
		By Definition~\eqref{eq:rational-surgery-correction},
		\begin{equation}\label{eq:difference-rational-correction}
			P_{\h,r}(\lambda)-P_{0,r}(\lambda) = -\frac{\CK M_0 \, h^2r^4}{4\lt(\lambda^2+\CK M_0\rt)^2} \, .
		\end{equation}
		On $\Gamma_1$, Equations~\eqref{eq:def_tildeps}, \eqref{eq:bound_convexity_mu_2}, and~\eqref{eq:centered-difference-endpoints-Gr} imply
		\begin{equation*}
			\n{\epsilon_{\mu,\h}^{-1}} + \n{\epsilon_{\mu,0}^{-1}} \leq C \, r^2 \, ,
			\qquad
			\n{\epsilon_{\mu,\h}-\epsilon_{\mu,0}} \leq C \, r^{-2} \min\{1,h^2r^2\} \, .
		\end{equation*}
		It follows that
		\begin{align*}
			\n{\frac{1}{\epsilon_{\mu,\h}}-\frac{1}{\epsilon_{\mu,0}}}
			&=\frac{\n{\epsilon_{\mu,\h}-\epsilon_{\mu,0}}}{\n{\epsilon_{\mu,\h}\epsilon_{\mu,0}}}\leq C \, r^2 \min\{1,h^2r^2\} \, .
		\end{align*}
		Moreover, since $\n{\lambda}\lesssim r$ on $\Gamma_1$ and $r$ is sufficiently small, $\n{\lambda^2+\CK M_0}\geq c$. Hence, because $hr\leq1$ in the low-frequency regime,
		\begin{equation*}
			\n{P_{\h,r}(\lambda)-P_{0,r}(\lambda)} \leq C \, h^2r^4 = C \, r^2\min\{1,h^2r^2\} \, .
		\end{equation*}
		Finally, since
		\begin{equation*}
			Q_{\h,r}-Q_{0,r} =\frac{1}{\epsilon_{\mu,\h}} - \frac{1}{\epsilon_{\mu,0}} + P_{\h,r}-P_{0,r},
		\end{equation*}
		we conclude that, for every $\lambda\in\Gamma_1$,
		\begin{equation}\label{eq:difference-post-residue-central-bound}
			\n{Q_{\h,r}(\lambda)-Q_{0,r}(\lambda)} \leq C \, r^2 \min\{1,h^2r^2\} \, .
		\end{equation}
		On $\Gamma_2\cup\Gamma_3$, we have
		\begin{equation*}
			Q_{\h,r}-Q_{0,r} = \frac{N_{\h,r}-N_{0,r}}{D_{\h,r}} + \frac{N_{0,r} \lt(D_{0,r}-D_{\h,r}\rt)}{D_{\h,r} \, D_{0,r}} \, .
		\end{equation*}
		The explicit expression for $N_{\h,r}$ obtained above gives
		\begin{align*}
			N_{\h,r}-N_{0,r}
			&=\frac{r^4 \lt(\zeta_{\h,r}-\zeta_{0,r}\rt)}{\lambda^2}
			\\
			&\quad + \frac{\CK M_0 \, r^4 \lt(\zeta_{\h,r}-\zeta_{0,r}\rt)
			+ \lt(b_\h-b_0\rt) \lt(b_\h+b_0\rt)}{\lambda^4}
			\\
			&\quad + \frac{r^4\bigl(\lt(b_\h-b_0\rt)\zeta_{\h,r}
			+b_0 \lt(\zeta_{\h,r}-\zeta_{0,r}\rt)\bigr)}{\lambda^6} \, .
		\end{align*}
		Therefore, using the inequalities
		\begin{equation*}
			\n{b_\h-b_0} = \frac{\CK M_0}{4} \, h^2 r^4 ,
			\qquad \n{b_\h}+ \n{b_0} \leq C \, r^2 ,
			\qquad \n{\zeta_{\h,r}-\zeta_{0,r}} \leq C \, h^2 r^2 ,
		\end{equation*}
		leads to
		\begin{equation}\label{eq:Nh-quantum-classical-difference}
			\n{N_{\h,r}-N_{0,r}} \leq C \left(\frac{r^4}{\n{\lambda}^2} + \frac{r^4}{\n{\lambda}^4} + \frac{r^6}{\n{\lambda}^6}\right) \min(1,h^2r^2) \, .
		\end{equation}
		For the term involving the difference of the denominators, we use
		\begin{equation*}
			D_{0,r}-D_{\h,r} = \lt(\epsilon_{\mu,0}-\epsilon_{\mu,\h}\rt) \lt(\lambda^2+\CK M_0\rt)^2,
		\end{equation*}
		together with~\eqref{eq:Lh-ray-expansion}--\eqref{eq:zeta-quantum-classical-difference}. Since $\n{\lambda}\geq R\, r$ on $\Gamma_2\cup\Gamma_3$, applying
		Inequality~\eqref{eq:ray-denominator-comparison} to both denominators
		gives, for every $\lambda\in\Gamma_2\cup\Gamma_3$,
		\begin{equation}\label{eq:difference-post-residue-ray-bound}
			\n{Q_{\h,r}(\lambda) - Q_{0,r}(\lambda)} \leq \frac{C\,r^4}{\n{\lambda}^2} \min(1, h^2r^2) \, .
		\end{equation}
		Inequalities~\eqref{eq:difference-post-residue-central-bound} and~\eqref{eq:difference-post-residue-ray-bound} are the difference analogues of inequalities~\eqref{eq:post-residue-central-bound} and~\eqref{eq:post-residue-ray-bound}, and we similarly deduce for $0<r\leq r_0$,
		\begin{equation*}
			\n{\hatGr_\h(t,k)-\hatGr_0(t,k)}
			\leq C \min(1, h^2r^2)\, r^3 e^{-\sigma rt}.
		\end{equation*}
		At non-small frequencies, one observes that for $\lambda = - \sigma r + i\omega$
		\begin{equation*}
			\n{\frac{1-\epsilon_{\mu,\h}}{\epsilon_{\mu,\h}} - \frac{1-\epsilon_{\mu,0}}{\epsilon_{\mu,0}}} \leq \frac{C\min(1,h^2r^2)}{r^2+\omega^2} \, .
		\end{equation*}
		Integrating in $\omega$ 
		\begin{equation*}
			\n{\hatGr_\h(t,k)-\hatGr_0(t,k)}
			\leq C \, r^{-1} \, e^{-\sigma rt} \min(1,h^2 r^2) \, .
		\end{equation*}
		On the compact intermediate annulus the same conclusion follows from the resolvent identity, uniform stability, and an integrable analytic majorant in $\omega$. Since $r^3/\weight{r}^4$ is comparable to $r^3$ at low frequency and to $r^{-1}$ at high frequency, these three estimates prove Inequality~\eqref{eq:splitr}.

		We finish by differentiating the same contour representations. Since $\mu$ and $K$ are radial, $\hatGr_\h(t,k)$ depends on $k$ only through $r=\n{k}$. For low frequencies, we write $\lambda=rz$ on each of the three pieces of $\Gamma$ and set
		\begin{equation*}
			\mathscr Q_\h(r,z):=Q_{\h,r}(rz) \, .
		\end{equation*}
		The rescaled contour $\widetilde\Gamma=r^{-1}\Gamma$ is independent of $r$. Its central part $r^{-1} \Gamma_1$ is bounded, while on its two rays $r^{-1} \Gamma_2$ and $r^{-1} \Gamma_3$, $\Re z\leq-\sigma - c\n{z}$ for some $c>0$.

		We claim that, for $0<r\leq r_0$, $0\leq j\leq2$, and $z\in\widetilde\Gamma$,
		\begin{align}\label{eq:differentiated-Q-low}
			\n{\partial_r^j\mathscr Q_\h(r,z)} &\leq C \, r^{2-j} \weight{z}^{-2},
			\\ \label{eq:differentiated-Q-difference-low}
			\n{\partial_r^j \bigl(\mathscr Q_\h-\mathscr Q_0\bigr)(r,z)} &\leq C \, h^2 r^{4-j} \weight{z}^{-2} .
		\end{align}
		We first prove these estimates for $j=0$. On the two rays of the rescaled contour, the change of variables $\lambda = r z$ and Inequality~\eqref{eq:post-residue-ray-bound} give
		\begin{equation*}
			\n{\mathscr Q_\h(r,z)} =\n{Q_{\h,r}(rz)} \leq C \, \frac{r^4}{\n{rz}^2} = C \, r^2 \n{z}^{-2} .
		\end{equation*}
		Similarly, Inequality~\eqref{eq:difference-post-residue-ray-bound} gives
		\begin{align*}
			\n{\mathscr Q_\h(r,z)-\mathscr Q_0(r,z)} &\leq C \, r^2\n{z}^{-2} \min\{1,h^2r^2\} \leq C \, h^2 r^4 \n{z}^{-2},
		\end{align*}
		where we used $hr\leq1$ in the low-frequency regime. Thus, combining inequalities~\eqref{eq:post-residue-central-bound} and~\eqref{eq:difference-post-residue-central-bound} with the preceding ray estimates, we obtain inequalities~\eqref{eq:differentiated-Q-low} and~\eqref{eq:differentiated-Q-difference-low} for $j=0$.

		We now differentiate with respect to $r$, keeping $z$ fixed. On the
		central part of $r^{-1}\Gamma$, Formula~\eqref{eq:def_tildeps} gives
		\begin{equation*}
			\epsilon_{\mu,\h}(rz,k)^{-1} =\frac{r^2}{r^2+\CK\,\widetilde\epsilon_{\mu,\h}(iz,r)}.
		\end{equation*}
		By Inequality~\eqref{eq:bound_convexity_mu_2}, the denominator is bounded away
		from zero uniformly in $h$ and $r$. The centered Taylor argument used
		in~\eqref{eq:centered-difference-endpoints-Gr} may also be
		differentiated with respect to $r$ and gives, for $0\leq j\leq2$,
		\begin{equation}\label{eq:differentiated-centered-difference}
			\n{\partial_r^j \bigl(D_{hr}F(z)-F'(z)\bigr)} \leq C \, h^2r^{2-j}.
		\end{equation}
		Differentiating the preceding formula for $\epsilon_{\mu,\h}^{-1}$ therefore gives a bound of order $r^{2-j}$, whereas the quantum--classical difference is bounded by $C\, h^2 r^{4-j}$. The same respective bounds hold for $P_{\h,r}-1$ and $P_{\h,r}-P_{0,r}$. This proves~\eqref{eq:differentiated-Q-low} and~\eqref{eq:differentiated-Q-difference-low} on the central part $r^{-1}\Gamma_1$.

		On the two rays, we differentiate the explicit expression for $Q_{\h,r}$. Inequalities~\eqref{eq:zeta-radial-derivatives} and~\eqref{eq:zeta-difference-radial-derivatives}, together with
		\begin{equation*}
			b_\h(r) = O(r^2) \, , \quad \text{ and } \quad b_\h(r)-b_0(r) = \frac{\CK M_0}{4} \, h^2r^4 ,
		\end{equation*}
		show that each derivative in $r$ costs at most one power of $r^{-1}$, while the decay $\weight{z}^{-2}$ is preserved. The differentiated resolvent denominators remain uniformly controlled because the rescaled rays stay at a fixed distance from the two limiting poles. This completes the proof of inequalities~\eqref{eq:differentiated-Q-low} and~\eqref{eq:differentiated-Q-difference-low}.

		After the change of variables $\lambda=rz$, Identity~\eqref{eq:Gr-corrected-contour} becomes
		\begin{equation}\label{eq:scaled-Gr-contour}
			\hatGr_\h(t,r) =\frac{r}{i}\int_{\widetilde\Gamma} e^{2\pi rtz}\mathscr Q_\h(r,z) \d z \, .
		\end{equation}
		We may differentiate this formula twice. If $\ell$ derivatives hit the exponential, the resulting factor is $(tz)^\ell$. The elementary ray estimate
		\begin{equation*} 
			x^\ell\int_1^\infty s^\ell\weight{s}^{-2}e^{-cxs} \d s
			\leq C\weight{x}^{\ell} ,
			\qquad \text{for } x>0, \text{ and } 0\leq\ell\leq 2 ,
		\end{equation*}
		together with Inequality~\eqref{eq:differentiated-Q-low} gives for any $0\leq j\leq2$
		\begin{equation}\label{eq:radial-Gr-derivatives-low}
			\n{\partial_r^j\hatGr_\h(t,r)} \leq C \, r^{3-j} \weight{rt}^{j} e^{-\sigma rt} \, .
		\end{equation}
		Applying the same calculation to the difference and using Inequality~\eqref{eq:differentiated-Q-difference-low} gives
		\begin{equation}\label{eq:radial-Gr-difference-derivatives-low}
			\n{\partial_r^j(\hatGr_\h-\hatGr_0)(t,r)} \leq C\, h^2 r^{5-j} \weight{rt}^{j} e^{-\sigma rt}.
		\end{equation}

		For $r\geq r_0^{-1}$ and $\omega\in\R$, let
		\begin{equation*}
			\mathscr R_\h(\lambda,r) := \frac{1-\epsilon_{\mu,\h}(\lambda,k)}{\epsilon_{\mu,\h}(\lambda,k)} \, ,
			\quad
			\text{ with } \lambda = -\sigma r + i\omega \, ,
		\end{equation*}
		and denote differentiation along this contour by
		\begin{equation*}
			\d_r:=\partial_r-\sigma\, \partial_\lambda \, .
		\end{equation*}
		On this contour, Equation~\eqref{eq:def_tildeps} and $\lambda = -\sigma r+i\omega$ give
		\begin{equation*}
			\epsilon_{\mu,\h}(\lambda,k)
			=1 + \CK \, r^{-2} \, \widetilde\epsilon_{\mu,\h}(-\sigma r+i\omega,r) \, .
		\end{equation*}
		The weighted holomorphic bounds for the centered quotient, after one or two differentiations at fixed $\omega$, give for any $0\leq j\leq2$
		\begin{align*}
			\n{\d_r^j\epsilon_{\mu,\h}} &\leq\frac{C\,r^{-j}}{r^2+\omega^2} \, ,
			\\
			\n{\d_r^j(\epsilon_{\mu,\h}-\epsilon_{\mu,0})} &\leq\frac{C\,h^2r^{2-j}}{r^2+\omega^2} \, .
		\end{align*}
		Indeed, for $hr\leq1$ the second line follows from the centered Taylor formula, while for $hr\geq1$ it follows by bounding the two quotients separately and using $1\leq h^2r^2$. Differentiating the resolvent identity and using the uniform gap now gives
		\begin{align}\label{eq:high-resolvent-derivatives}
			\n{\d_r^j\mathscr R_\h(-\sigma r+i\omega,r)}
			&\leq\frac{C \, r^{-j}}{r^2+\omega^2} \, ,
			\\ \label{eq:high-resolvent-difference-derivatives}
			\n{\d_r^j(\mathscr R_\h-\mathscr R_0)(-\sigma r+i\omega,r)}
			&\leq\frac{C \, h^2r^{2-j}}{r^2+\omega^2} \,.
		\end{align}
		To see the powers of $r$ explicitly, observe that each application of $\d_r$ differentiates either the Coulomb factor $r^{-2}$, the centered difference with step $hr$, or the phase velocity
		\begin{equation*}
			\vartheta=\frac{i(-\sigma r+i\omega)}r = -i\sigma-\frac{\omega}{r} \, .
		\end{equation*}
		The first two operations cost at most one factor $r^{-1}$. The last produces $\omega/r^2$, but also differentiates a Cauchy kernel; the additional denominator compensates the factor $\omega/r$ and again
		leaves at most one factor $r^{-1}$. For the quantum--classical difference, the factor $h^2r^2$ comes from the centered Taylor estimate~\eqref{eq:differentiated-centered-difference}. Derivatives in $\lambda$ only improve the decay in $\omega$, so the right-hand
		sides above remain valid. This also supplies an integrable majorant and justifies differentiating the vertical-contour integral.

		When a derivative hits the exponential it gives the factor $-2\pi\sigma t$. Hence Leibniz' rule and~\eqref{eq:high-resolvent-derivatives}--\eqref{eq:high-resolvent-difference-derivatives} yield for any $0\leq j\leq2$,
		\begin{align*}
			\n{\partial_r^j\hatGr_\h(t,r)}
			&\leq C\, r^{-1-j}\weight{rt}^{j}e^{-\sigma rt},
			\\
			\n{\partial_r^j(\hatGr_\h-\hatGr_0)(t,r)}
			&\leq C \, h^2r^{1-j}\weight{rt}^{j}e^{-\sigma rt} .
		\end{align*}
		On the compact annulus the same bounds follow from smooth dependence on $r$, the uniform resolvent gap, and compactness.

		Finally, if $u(k) = u_0(\n{k})$ is radial, then for $1\leq\n{\alpha}\leq2$,
		\begin{equation*}
			\n{\partial_k^\alpha u(k)}
			\leq C\sum_{j=1}^{\n{\alpha}}
			r^{j-\n{\alpha}}\n{\partial_r^j u_0(r)} .
		\end{equation*}
		Combining this identity with the low-, intermediate-, and high-frequency radial estimates proves inequalities~\eqref{eq:Gr-symbol-derivatives} and~\eqref{eq:Gr-difference-symbol-derivatives}.
	\end{proof}


\subsection{Green functions in the physical space \label{physicalspace}}


	We now bound $\Gosc_\pm$ and $\Gr$ in the physical space.
	
	\begin{lem}\label{lem:decay_Gosc}
		For every $t>0$, uniformly in $\h\in[0,1]$,
		\begin{equation} \label{eq:decayGosc1}
			\Nrm{\Gosc_{\h,\pm}(t,\cdot)}{L^\infty} \lesssim \weight{t}^{-3/2} \quad \text{ and } \quad \Nrm{\Gosc_{\h,\pm}(t,\cdot)}{L^2} \lesssim 1 \, .
		\end{equation}
		Moreover,
		\begin{align}\notag
			\Nrm{\Gosc_{\h,\pm}(t,\cdot)-\Gosc_{0,\pm}(t,\cdot)}{L^\infty}
			&\lesssim h^2,
			\\\label{eq:decayGosc2}
			\Nrm{\Gosc_{\h,\pm}(t,\cdot)-\Gosc_{0,\pm}(t,\cdot)}{L^2}
			&\lesssim \min(1,h^2\weight{t}) \, .
		\end{align}
	\end{lem}

	\begin{proof}
		Let $0 \le \h \le 1$. We recall that
		\begin{equation}\label{eq:Gsreal}
			\Gosc_{\h,\pm}(t,x) = \int e^{2\pi\lambda_{\h,\pm}(k) t + 2i\pi k \cdot x} \,a_{\h,\pm}(k) \d k \,.
		\end{equation}
		Note first that, as $a_{\h,\pm}(k)$ is compactly supported and smooth, $\Gosc_{\h,\pm}(t,x)$ is uniformly bounded in $t$ and $x$ and rapidly decaying as $x$ goes to infinity. In particular, inequalities~\eqref{eq:decayGosc1} and~\eqref{eq:decayGosc2} are true for $0 \le t \le 1$.
		
		Moreover, as soon as $k \ne 0$, 
		$\Re{\lambda_{\h,\pm}(k)} < 0$, and both $\lambda_{\h,\pm}$ and $a_{\h,\pm}$ are holomorphic in the ball of center $0$ and radius $\kappa$, provided $\kappa > 0$ is small enough. The integral on $\n{k} \ge \kappa/2$ is thus exponentially decreasing in time. Let us focus on the integral on $\n{k} < \kappa/2$. As both $\lambda_{\h,\pm}$ and $a_{\h,\pm}$ are holomorphic in this area, we may use the steepest descent method. For large $t$, the right-hand side of Equation~\eqref{eq:Gsreal} is an oscillatory integral in $k$, with phase $t\, \Theta_{\h,\pm}(k,y)$, where
		\begin{equation*}
			\Theta_{\h,\pm}(k,y) = \lambda_{\h,\pm}(k) + i\, k \cdot y
		\end{equation*}
		where $y = x/t$. 
		The phase is stationary at $k=k_s$ such that $\nabla_k \Theta_{\h,\pm}(k_s,y) = 0$, namely when
		\begin{equation*}
			\nabla_k \lambda_{\h,\pm}(k_s) + i \,y = 0 \, .
		\end{equation*}
		We recall that the real part of $\nabla_k \lambda_{\h,\pm}(k,s)$ is of order $k^{N-1}$ for polynomial equilibrium, and even smaller for exponential ones, whereas its imaginary part equals $2 k + O(\n{k}^3)$. We expand the equation in $k$, which leads to 
		\begin{equation*}
			\pm 2 k_s \lt(1+ O(\n{k_s}^2) \rt) + y = O(k_s^3\, h^2)
		\end{equation*}
		or equivalently
		\begin{equation} 
			k_s = \mp \frac{y}{2} \lt(1 + O(\n{y}^2) + O(\n{y}^3 h^2) \rt) ,
		\end{equation}
		which gives the following asymptotic expansion
		\begin{equation*}
			\Gosc_{\h,\pm}(t,x) \sim \frac{C_0}{t^{\frac{3}{2}}} \exp\!\( \Lambda_{\h,\pm}(\n{y}^2)\, t \) ,
		\end{equation*}
		for some constant $C_0$, where
		\begin{equation*}
			\Lambda_{\h,\pm}(u) = \lambda_{\h,\pm}\!\lt(u^{1/2}\rt)
			= \re \lambda_{\h,\pm}\!\lt(u^{1/2}\rt) \pm i \mp i \, u \lt( 1+ O( u ) \rt) + O(h^2)
		\end{equation*}
		and where 
		\begin{equation*}
			y = \frac{x}{2 t} + O \Bigl( \frac{ \n{x}^3}{t^3} \Bigr) \, .
		\end{equation*}
		Using $\n{\Lambda_{\h,\pm}(u) - \Lambda_{0,\pm}(u)}\lesssim h^2$ and $\n{a_{\h,\pm}(k) - a_{0,\pm}(k)} \lesssim h^2$, we obtain
		\begin{align*}
			\Nrm{\Gosc_{\h,\pm}(t,\cdot) - \Gosc_{0,\pm}(t,\cdot)}{L^\infty}
			& \lesssim e^{\re \Lambda_{\h,\pm}(|y|^2 t)} \n{a_{\h,\pm}- a_{0,\pm}}
			\\
			& \quad + | e^{\Lambda_{\h,\pm}(|y|^2 t)} - e^{\Lambda_{0,\pm}(|y|^2 t)}| \n{a_{0,\pm}} .
		\end{align*}
		For $t \le h^{-2}$,
		\begin{equation*}
			\Nrm{\Gosc_{\h,\pm}(t,\cdot) - \Gosc_{0,\pm}(t,\cdot)}{L^\infty} \lesssim \frac{e^{h^2 t} -1}{\weight{t}^{3/2}} + h^2 \lesssim h^2 \, ,
		\end{equation*}
		and for $t \ge h^{-2}$, 
		\begin{equation*}
			\Nrm{\Gosc_{\h,\pm}(t,\cdot) - \Gosc_{0,\pm}(t,\cdot)}{L^\infty} \lesssim \frac{2}{\weight{t}^{3/2}} + h^2 \lesssim h^2 .
		\end{equation*}
		The bound in $L^2$ follows similarly since $a$ is rapidly decaying in $x$ using the smallness of the $\n{\lambda_\h-\lambda_0}$, $\n{a_\h-a_0}$ and Plancherel theorem. 
	\end{proof}

	We now turn to bounds on the remainder $\Gr_h$. We define
	\begin{equation*}
		\GrE_h = - \nabla_x \Delta^{-1} \Gr_h
	\end{equation*}
	to be the remainder Green function associated with the electric field.
	
	\begin{lem}\label{lem:decay_Gr}
		For any $h\in[0,1]$ and $t>0$, it holds that
		\begin{align*}
			\Nrm{\Gr_\h(t,\cdot)}{L^\infty} &\lesssim \frac{1}{t^2\weight{t}^4} \,  ,
			&
			\Nrm{\Gr_\h(t,\cdot)}{L^1} &\lesssim \frac{1}{t^{1/8}\weight{t}^{3-1/8}} \,  ,
			\\
			\Nrm{\GrE_\h(t,\cdot)}{L^\infty} &\lesssim \frac{1}{t\weight{t}^4} \,  ,
			&
			\Nrm{\GrE_\h(t,\cdot)}{L^1} &\lesssim \frac{1}{\weight{t}^{2}} \, . 
		\end{align*}
		Moreover, we have
		\begin{align*}
			\Nrm{\Gr_\h(t,\cdot) - \Gr_0(t,\cdot)}{L^\infty} &\lesssim \frac{h^2}{t^4\weight{t}^4} \, ,
			&
			\Nrm{\Gr_\h(t,\cdot) - \Gr_0(t,\cdot)}{L^1} &\lesssim \frac{h^{2}}{t\weight{t}^{4}} 
			\\
			\Nrm{\GrE_\h(t,\cdot) - \GrE_0(t,\cdot)}{L^\infty} &\lesssim \frac{h^2}{t^3\weight{t}^4} \, ,
			&
			\Nrm{\GrE_\h(t,\cdot) - \GrE_0(t,\cdot)}{L^1} &\lesssim \frac{h^{2}}{t^{3/8}\weight{t}^{4-3/8}} \, .
		\end{align*}
	\end{lem}

	\begin{proof}
		We first record two elementary scaling estimates. For $m>-3$ and $t>0$, we have
		\begin{equation}\label{eq:radial-L1-scaling}
			\intd \n{k}^m e^{-\sigma \n{k}t} \d k = \tfrac{4\pi \Gamma(m+3)}{\sigma^{m+3}}\, t^{-m-3} \, .
		\end{equation}
		Similarly, if $m-a \geq - 3/2$, then
		\begin{equation}\label{eq:radial-L2-scaling}
			\bNrm{\n{k}^{m-a}\weight{\n{k}t}^{a}e^{-\sigma \n{k} t}}{L^2} = C\, t^{-m+a-3/2}
		\end{equation}
		for some finite constant $C>0$ depending on $a$, $k$ and $\sigma$. On the other hand, when $t\to 0$, then by dominated convergence, for $p\geq 1$
		\begin{align*} 
			&\Nrm{\weight{k}^{-4}\n{k}^{m-a}\weight{\n{k}t}^{a}e^{-\sigma \n{k} t}}{L^p} \to C &&\text{ if } -\frac{3}{p}<m-a<4-\frac{3}{p} \, ,
            \\
            &\Nrm{\weight{k}^{-4}\n{k}^{m-a}\weight{\n{k}t}^{a}e^{-\sigma \n{k} t}}{L^p} \sim C\,t^{4-3/p+a-m} &&\text{ if } m-a>4-\frac{3}{p} \, .
		\end{align*}

		Let
		\begin{equation*}
			\mathcal M_E(k) := -\widehat{\nabla K}(k) = -2i\pi \CK \, \frac{k}{\n{k}^2} \, ,
			\qquad
			\hatGrE_\h(t,k)=\mathcal M_E(k)\hatGr_\h(t,k) \, .
		\end{equation*}
		Since $\mathcal M_E$ is homogeneous of degree $-1$, for any multi-index $\beta$ one has
		\begin{equation*}
			\n{\partial_k^\beta\cM_E(k)} \leq C_\beta \n{k}^{-1-\n{\beta}}  .
		\end{equation*}
		Leibniz' rule and the differentiated bounds~\eqref{eq:Gr-symbol-derivatives} and~\eqref{eq:Gr-difference-symbol-derivatives} therefore imply, for $\n{\alpha}\leq 2$,
		\begin{align}\label{eq:GrE-symbol-derivatives}
			\n{\partial_k^\alpha \hatGrE_\h(t,k)}
			&\lesssim \frac{\n{k}^{2-\n{\alpha}}}{\weight{k}^4}
			\weight{\n{k}t}^{\n{\alpha}}e^{-\sigma \n{k}t},
			\\\label{eq:GrE-difference-symbol-derivatives}
			\n{\partial_k^\alpha
			\bigl(\hatGrE_\h - \hatGrE_0\bigr)(t,k)}
			&\lesssim h^2 \, \frac{\n{k}^{4-\n{\alpha}}}{\weight{k}^4}
			\weight{\n{k}t}^{\n{\alpha}}e^{-\sigma \n{k}t}.
		\end{align}
		Indeed, every term in the first Leibniz sum is bounded by
		\begin{equation*}
			\n{k}^{-1-|\beta|}
			\frac{\n{k}^{3-|\alpha-\beta|}}{\weight{k}^4}
			\weight{\n{k}t}^{|\alpha-\beta|}e^{-\sigma \n{k}t}
			\leq \frac{\n{k}^{2-\n{\alpha}}}{\weight{k}^4}
			\weight{\n{k}t}^{\n{\alpha}}e^{-\sigma \n{k}t},
		\end{equation*}
		and the estimate on the difference is identical with $\n{k}^3$ replaced by $h^2\n{k}^5$. Together with inequalities~\eqref{eq:Gr-symbol-derivatives} and~\eqref{eq:Gr-difference-symbol-derivatives}, these estimates show that the derivatives up to order two of all four symbols are locally square integrable at $k=0$. 

		We now prove the $L^\infty$ bounds. The Fourier inversion formula gives 
		\begin{align*}
			\Nrm{\Gr_\h(t,\cdot)}{L^\infty} \leq\sNrm{\hatGr_\h(t,\cdot)}{L^1} &\lesssim t^{-6},
			\\
			\Nrm{\GrE_\h(t,\cdot)}{L^\infty} \leq\sNrm{\widehat{\GrE_\h}(t,\cdot)}{L^1} &\lesssim t^{-5},
			\\
			\Nrm{\Gr_\h(t,\cdot)-\Gr_0(t,\cdot)}{L^\infty} &\lesssim \h^2\, t^{-8},
			\\
			\Nrm{\GrE_\h(t,\cdot)-\GrE_0(t,\cdot)}{L^\infty} &\lesssim \h^2 \, t^{-7}.
		\end{align*}
		It remains to establish the $L^1$ bounds. We shall use the following weighted interpolation inequality in dimension three
		\begin{equation}\label{eq:weighted-L2-to-L1}
			\Nrm{g}{L^1} \lesssim \Nrm{g}{L^2}^{1/4} \Nrm{\n{x}^2g}{L^2}^{3/4} ,
		\end{equation}
		which follows for instance by writing for $R>0$, by the Cauchy--Schwarz inequality,
		\begin{align*}
			\Nrm{g}{L^1} = \int_{B_R} \n{g} \, dx + \int_{\Rd\setminus B_R} \n{g} \, dx 
			\lesssim R^{3/2}\Nrm{g}{L^2} + R^{-1/2}\Nrm{\n{x}^2g}{L^2}
		\end{align*}
		and optimizing in $R$. By Plancherel's theorem
		\begin{equation}\label{eq:weighted-plancherel}
			\Nrm{\n{x}^2g}{L^2}
			=\frac{1}{4\pi^2} \Nrm{\Delta_k\widehat g}{L^2}
			\lesssim
			\sum_{\n{\alpha}=2}
			\Nrm{\partial_k^\alpha\widehat g}{L^2}.
		\end{equation}
		Applying Formula~\eqref{eq:radial-L2-scaling} to \eqref{eq:Gr-symbol-derivatives} and~\eqref{eq:GrE-symbol-derivatives} yields
		\begin{align*}
			\Nrm{\Gr_\h(t)}{L^2} &\lesssim t^{-9/2},
			&\Nrm{\n{x}^2\Gr_\h(t)}{L^2}
			&\lesssim t^{-5/2},\\
			\Nrm{\GrE_\h(t)}{L^2}
			&\lesssim t^{-7/2},
			&\Nrm{\n{x}^2\GrE_\h(t)}{L^2}
			&\lesssim t^{-3/2}.
		\end{align*}
		Hence the interpolation inequality~\eqref{eq:weighted-L2-to-L1} gives
		\begin{equation*}
			\Nrm{\Gr_\h(t)}{L^1}\lesssim t^{-3},
			\qquad
			\Nrm{\GrE_\h(t)}{L^1}\lesssim t^{-2}.
		\end{equation*}
		For $0<t\leq1$, retain the factor $\weight{k}^{-4}$ in~\eqref{eq:GrE-symbol-derivatives} instead of discarding it in the scaling estimate. The cases $\n{\alpha}=0$ and $\n{\alpha}=2$ then give
		\begin{equation*}
			\Nrm{\GrE_\h(t)}{L^2} + \Nrm{\n{x}^2\GrE_\h(t)}{L^2}\leq C \, .
		\end{equation*}
		Indeed, $\weight{t k}^2 = 1+t^2\n{k}^2$, and both resulting radial integrals are bounded uniformly for $0<t\leq1$ after the change of variables $r=t\n{k}$ in the term containing $t^2$. Thus \eqref{eq:weighted-L2-to-L1} also gives $\sNrm{\GrE_\h(t)}{L^1}\leq C$ on this interval. Similarly, applying Formula~\eqref{eq:radial-L2-scaling} and Inequality~\eqref{eq:weighted-plancherel} to the two differences of symbols gives
		\begin{align*}
			\Nrm{\Gr_\h(t)-\Gr_0(t)}{L^2} &\lesssim \h^2 \, t^{-13/2},
			&\Nrm{\n{x}^2(\Gr_\h-\Gr_0)(t)}{L^2} &\lesssim \h^2 \, t^{-9/2},
			\\
			\Nrm{\GrE_\h(t)-\GrE_0(t)}{L^2} &\lesssim \h^2 \, t^{-11/2},
			&\Nrm{\n{x}^2(\GrE_\h-\GrE_0)(t)}{L^2} &\lesssim \h^2 \, t^{-7/2}.
		\end{align*}
		A last use of \eqref{eq:weighted-L2-to-L1} gives
		\begin{equation*}
			\Nrm{\Gr_\h(t)-\Gr_0(t)}{L^1} \lesssim \h^2 \, t^{-5},
			\qquad
			\Nrm{\GrE_\h(t)-\GrE_0(t)}{L^1} \lesssim \h^2 \, t^{-4}.
		\end{equation*}
		This completes the proof.
	\end{proof}


\section{Proof of the main theorems}


	\begin{proof}[Proof of Theorem~\ref{thm:Green_functions} and Theorem~\ref{thm:Green_functions_limit}]
		These two theorems are now a direct consequence of the two previous lemmas, together with the formulas~\eqref{eq:inversion2} and~\eqref{eq:hatG_expansion}.
	\end{proof}

	\begin{proof}[Proof of Theorem~\ref{thm:simplified}]
		Formula~\eqref{eq:limit_rho} follows directly from Theorem~\ref{thm:Green_functions} and Theorem~\ref{thm:Green_functions_limit}. The proof of \eqref{eq:decay_E} is exactly the same as in the classical case and we refer to~\cite{bedrossian_linearized_2022} for detailed proofs. More precisely, it implies for $E= E_\h$ and $E=E_0$ that for any $q\in[2,\infty]$,
		\begin{equation*}
			\Nrm{E}{L^q} \lesssim \weight{t}^{-3\lt(\frac{1}{2}-\frac{1}{q}\rt)} .
		\end{equation*}
		We now turn to the study of $E_0 - E_h$, which equals
		\begin{align*}
			E_0 - E_h &= (\GoscE_{0,+} - \GoscE_{h,+}) \star_{t,x} \rho^\circ + (\GoscE_{0,-} - \GoscE_{h,-}) \star_{t,x} \rho^\circ
			\\
			&\quad + (\GrE_0 - \GrE_h) \star_{t,x} \rho^\circ + (\GrE_h + \GoscE_{h,+} +\GoscE_{h,-}) \star_{t,x} (\rho^\circ - \rho^\circ_h)
		\end{align*}
		where $\rho^\circ$ is the solution of the free transport with initial data $f_0^\init$, and $\rho^\circ_h$ the one associated to $f_h^\init$. Changes of variables and Minkowski's inequality give the following standard dispersive estimate for the free transport equation
		\begin{equation*}
			\Nrm{\rho^\circ(\tau,\cdot)}{L^2} \leq \max(1,\tau)^{-3/2}\, \sNrm{f_0^\init}{L^1_xL^2_v\cap L^1_vL^2_x} \, .
		\end{equation*}
		thus
		\begin{align*}
			\sNrm{(\GrE_0 - \GrE_h) \star_{t,x} \rho^\circ}{L^2} 
			\leq \int_0^t \sNrm{\GrE_0(\tau,\cdot) - \GrE_h(\tau,\cdot)}{L^1} 
			\Nrm{\rho^\circ(t - \tau,\cdot)}{L^2} \d\tau
			\\
			\lesssim \sNrm{f_0^\init}{L^1_xL^2_v\cap L^1_vL^2_x} \int_0^t \frac{h^2 \d \tau}{\weight{\tau}^{4-3/8} \tau^{3/8} \weight{t - \tau}^{3/2}}  \lesssim \frac{h^2}{t^{3/8}\weight{t}^{9/8}} .
		\end{align*}
        
		It remains to bound the oscillatory Green functions. We have
		\begin{align*}
			(\GoscE_{0,+} - \GoscE_{h,+}) \star_{t,x} \rho^\circ
			&= \int_0^t \lt( \GoscE_{0,+}(t-\tau,\cdot) - \GoscE_{h,+}(t - \tau,\cdot) \rt) \underset{x}{*} \rho^\circ(\tau,\cdot) \d\tau
			\\
			&= \lt( G^{\textnormal{osc},E,1}_{0,+,1}(t-\tau,\cdot) - G^{\textnormal{osc},E,1}_{h,+}(t - \tau,\cdot) \rt) \underset{x}{*} \rho^\circ(\tau,\cdot)\Bigr|_0^t
			\\
			&\qquad - \int_0^t \lt( \GoscE_{0,+}(t-\tau,\cdot) - G^{\textnormal{osc},E,1}_{h,+}(t - \tau,\cdot) \rt) \underset{x}{*} \partial_\tau \rho^\circ(\tau,\cdot) \d\tau
		\end{align*}
		where $G^{\textnormal{osc},E,1}_{0,\pm}$ through its Fourier transform by
		\begin{equation*}
			\widehat G^{\textnormal{osc},E,1}_{0,\pm} = \frac{\hatGoscE_{0,\pm}}{2 \pi \lambda_\pm(k)}
		\end{equation*}
		so that by definition of $\Gosc$
		\begin{equation*}
			\partial_t G^{\textnormal{osc},E,1}_{0,\pm} = \GoscE_{0,\pm}, 
		\end{equation*}
		and similarly for $G^{\textnormal{osc},E,1}_{h,\pm}$. Using the equation of conservation
		\begin{equation*}
			\partial_t \rho^\circ = - \Dx \cdot m^\circ
		\end{equation*}
		where $m^\circ$ is the first moment of $f_0^\circ$, defined by
		\begin{equation*}
			m^\circ(t,x) = \intd f_0^\circ(t,x,v) \, v \d v \, ,
		\end{equation*}
		and integrating by parts gives
		\begin{align*}
			&\int_0^t \Bigl( G^{\textnormal{osc},E,1}_{0,+}(t-\tau,\cdot) - G^{\textnormal{osc},E,1}_{h,+}(t - \tau,\cdot) \Bigr) \underset{x}{*} \partial_\tau \rho^\circ(\tau,\cdot) \d\tau
			\\
			&= - \int_0^t \Bigl( \Dx G^{\textnormal{osc},E,1}_{0,+}(t-\tau,\cdot) - \Dx G^{\textnormal{osc},E,1}_{h,+}(t - \tau,\cdot) \Bigr) \underset{x}{*} m^\circ(\tau,\cdot) \d\tau \, .
		\end{align*}
		We observe that the symbol of $\Dx G^{\textnormal{osc},E,1}_{0,\pm}$ has the same properties as $\Gosc_{0,\pm}$. As a consequence, the $L^2$ norm of the integral is bounded by
		\begin{equation*}
			\int_0^t h^2 \Nrm{m^\circ(\tau,\cdot)}{L^1} \d\tau \leq \sNrm{\n{v} f_0^\init}{L^1} \, h^2 \, t \, .
		\end{equation*}
		It remains to bound the terms at $0$ and $t$. We have
		\begin{equation*}
			\bigl(G^{\textnormal{osc},E,1}_{0,+} - G^{\textnormal{osc},E,1}_{h,+}\bigr)(t,\cdot) \underset{x}{*} \rho^\circ(0,\cdot)
			= - \bigl(G^{\textnormal{osc},1}_{0,+}(t,\cdot) - G^{\textnormal{osc},1}_{h,+}(t,\cdot)\bigr)(t,\cdot) \underset{x}{*} \nabla \Delta^{-1} \rho^\circ(0,\cdot)
		\end{equation*}
		where $G^{\textnormal{osc},1}_{0,\pm}$ is defined by $\partial_t G^{\textnormal{osc},1}_{0,\pm} = \Gosc_{0,\pm}$ and similarly for $G^{\textnormal{osc},1}_{h,\pm}$. We note that $G^{\textnormal{osc},1}_{0,\pm}$ and $\Gosc_{0,\pm}$ share the same bounds, thus the $L^2$ norm of the previous convolution is bounded by $h^2 \Nrm{\nabla \Delta^{-1} \rho^\circ(0,\cdot)}{L^1}$, namely by $h^2$ provided the electric field associated with the initial data is $L^1$.
		
		To bound
		\begin{equation*}
			\bigl( G^{\textnormal{osc},E,1}_{0,+}(0,\cdot) - G^{\textnormal{osc},E,1}_{h,+}(0,\cdot) \bigr) \underset{x}{*} \rho^\circ(t,\cdot)
			+ \bigl( G^{\textnormal{osc},E,1}_{0,-}(0,\cdot) - G^{\textnormal{osc},E,1}_{h,-}(0,\cdot) \bigr) \underset{x}{*} \rho^\circ(t,\cdot) \, , 
		\end{equation*}
		we observe that we have a particular cancellation
		\begin{align*}
			\hatG^{\textnormal{osc},1}_{0,+}(0,k) + \hatG^{\textnormal{osc},1}_{0,-}(0,k)
			&= \frac{a_+(k)}{\lambda_+(k)} + \frac{a_-(k)}{\lambda_-(k)}
			\\
			&= \frac{{i\over 2}+O(|k|^2)}{i+O(|k|^2)} + \frac{-{i\over 2}+O(|k|^2)}{-i+O(|k|^2)} = 1+O(\n{k}^2) \, ,
		\end{align*}
		thus
		\begin{equation*}
			\widehat G^{\textnormal{osc},1}_{0,+}(0,k) + \widehat G^{\textnormal{osc},1}_{0,-}(0,k) - \widehat G^{\textnormal{osc},1}_{h,+}(0,k) - \widehat G^{\textnormal{osc},1}_{h,-}(0,k) = O(|k|^2 \, h^2) \, .
		\end{equation*}
		Thus this last term is bounded by $h^2$.

		Now that we obtained a bound on the difference between the electric fields, we are ready to estimate the difference between the phase-space densities. We observe that
		\begin{equation*}
			(\dpt + v\cdot\Dx)(f-f_\h) + (E_0 -E_h) \cdot\Dv\mu = E_\h \underset{x}{*} \cF_k\lt( \Dv\mu - D_{h k}\mu\rt)
		\end{equation*}
		and deduce by Gr\"onwall's lemma and H\"older's inequality that
		\begin{equation*}
			\sNrm{f_0(t,\cdot,\cdot)-f_\h(t,\cdot,\cdot)}{L^2} \leq \sNrm{f_0^\init-f_\h^\init}{L^2} + \int_0^t J_1(s) + J_2(s) \d s
		\end{equation*}
		with
		\begin{align*}
			J_1(t) &= \Nrm{E_0(t,\cdot)-E_\h(t,\cdot)}{L^2} \Nrm{\Dv\mu}{L^2} \, ,
			\\
			J_2(t)^2 &= \intdd \n{\widehat{E}_\h(t,k) \cdot\lt(\Dv\mu(v) - D_{h k}\mu(v)\rt)}^2\d k\d v \, .
		\end{align*}
		The term $J_1$ is estimated by $\h^2\weight{t}$, while for $J_2$, the regularity assumptions on $\mu$ imply that 
		\begin{equation*}
			J_2(t)^2 \leq C_\mu\,h^4 \intdd \n{\widehat{E}_\h(t,k)}^2 \n{2\pi k}^4 \d k\d v = C_\mu\,h^4 \Nrm{\nabla\rho_\h(t,\cdot)}{L^2}^2 .
		\end{equation*}
        Similarly as in Lemma~\ref{lem:decay_Gr}, one obtains
        \begin{equation*}
            \Nrm{\Dx \Gr_\h}{L^1} \lesssim \frac{1}{t^{3/8}\weight{t}^{4-3/8}}
        \end{equation*}
        while $\Dx\Gosc_{\h,\pm}$ can be treated similarly as $\Gosc_{\h,\pm}$ 
        and a standard dispersive estimate for the free Schrödinger equation gives
        \begin{equation*}
			\Nrm{\rho^\circ(\tau,\cdot)}{L^2} \lesssim \weight{\tau}^{-3/2} \lt(\Nrm{\weight{x}^{3/2}\op^\init}{\L^2} + \Nrm{\weight{\opp}^{3/2}\op^\init}{\L^2}\rt) ,
		\end{equation*}
        from which one deduces a similar bound for $J_2$ as for $J_1$.
	\end{proof}

\subsubsection*{Acknowledgment}
D. Bian was partially supported by the NSF of China (Grant No. 12271032). Quoc-Hung Nguyen's research was supported by the CAS Project for Young
	Scientists in Basic Research (Grant No.~YSBR-031) and by the NSF of China (Grant Nos.~1251101538 and
	12595282).


\bibliographystyle{abbrv} 
\bibliography{Vlasov}

\end{document}